\documentclass[12pt,reqno]{amsart}

\usepackage[a4paper]{anysize}\marginsize{2.1cm}{2.1cm}{2.1cm}{2.1cm}
\pdfpagewidth=\paperwidth \pdfpageheight=\paperheight
\usepackage{amsfonts,amssymb,amsthm,amsmath}
\usepackage{eucal}
\usepackage{verbatim}
\usepackage[all]{xy}
\usepackage[english]{babel}
\usepackage{graphicx}
\usepackage{pgf}
\usepackage{array}
\usepackage{tikz}
\usetikzlibrary{shapes,arrows}
\usepackage{float}
\restylefloat{table}
\usepackage{multicol}
\usepackage{multirow}
\usepackage[foot]{amsaddr}
\usepackage{hyperref}
\hypersetup{
    colorlinks,
    linkcolor=black,
    citecolor=black,
    urlcolor=black,
    breaklinks=true
}

\theoremstyle{plain}
\newtheorem{thm}{Theorem}[section]
\newtheorem{lemma}[thm]{Lemma}
\newtheorem{proposition}[thm]{Proposition}
\newtheorem{corollary}[thm]{Corollary}
\newtheorem{conjecture}[thm]{Conjecture}

\theoremstyle{definition}
\newtheorem{definition}[thm]{Definition}
\newtheorem{remark}[thm]{Remark}
\newtheorem{example}[thm]{Example}

\newcommand\N{\mathbb N}
\newcommand\f{\varphi}
\renewcommand\P{\mathbb P}

\newcommand\newop[2]{\def#1{\mathop{\rm #2}\nolimits}}
\newop\Bs{Bs}
\newop\edim{expdim}
\newop\vdim{vdim}
\newcommand\p{\mathbb{P}}

\begin{document}

\author{\L ucja Farnik}
\address{Department of Mathematics, Pedagogical University of Cracow,
   Podchor\c a\.zych 2,
   PL-30-084 Krak\'ow, Poland, \href{http://orcid.org/0000-0003-4300-0206}{ORCiD: 0000-0003-4300-0206}}
\email{Lucja.Farnik@gmail.com}

\author{Francesco Galuppi}
\address{Max Planck Institute for Mathematics in the Sciences, Inselstra\ss{}e 22, 04103 Leipzig, Germany, \href{http://orcid.org/0000-0001-5630-5389}{ORCiD: 0000-0001-5630-5389}}
\email{galuppi@mis.mpg.de}

\author{Luca Sodomaco}
\address{Department of Mathematics, University of Florence, Viale Morgagni, 67/a, 50134 Florence, Italy, \href{http://orcid.org/0000-0002-0472-6357}{ORCiD: 0000-0002-0472-6357}}
\email{luca.sodomaco@unifi.it}

\author{William Trok}
\address{Department of Mathematics, University of Kentucky,  715 Patterson  Office Tower,  Lexington, KY 40506-0027, USA}
\email{william.trok@uky.edu}


\title{On the unique unexpected quartic in $\mathbb{P}^2$}

\begin{abstract}
The computation of the dimension of linear systems of plane curves through a bunch of given multiple points is one of the most classic  
issues in Algebraic Geometry. In general, it is still an open problem to understand when the points fail to impose independent conditions. Despite many partial results, a complete solution is not known, even if the fixed points are in general position. The answer in the case of general points in
the projective plane is predicted by the famous Segre-Harbourne-Gimigliano-Hirschowitz
conjecture.

When we consider fixed points in special position, even more interesting situations may occur. Recently Di Gennaro, Ilardi and Vall\`{e}s discovered a special configuration $Z$ of nine points with a remarkable property: a general triple point always fails to impose independent conditions on the ideal of $Z$ in degree four. The peculiar structure and properties of this kind of \textit{unexpected curves} were studied by Cook II, Harbourne, Migliore and Nagel. 

By using both explicit geometric constructions and more abstract algebraic arguments, we classify low degree unexpected curves. In particular, we prove that the aforementioned configuration $Z$ is the unique one giving rise to an unexpected quartic.
\end{abstract}

\maketitle

\let\thefootnote\relax\footnote{The research of \L ucja Farnik was partially supported by the National Science Centre, Poland, grant 2018/28/C/ST1/00339. Luca Sodomaco is a member of INDAM-GNSAGA.}

\thispagestyle{empty}

\section{Introduction}\label{introduction}
One of the central problems in Algebraic Geometry is the study of linear systems of hypersurfaces of $\p^n$ with imposed singularities, namely divisors containing a set of given points $P_1,\dots,P_r$ with multiplicities $m_1,\dots,m_r$. Interpolation Theory addresses the problem of computing the dimension of such systems. It is actually an open problem to understand when the points fail to impose independent conditions --- naively counting parameters does not always give the correct result.

Let us start with the basic definitions in the projective plane. We work over the field of complex numbers $\mathbb{C}$.

\begin{definition}
Given $P_1,\dots,P_r\in\P^2$ and their ideals $I_1,\dots, I_r\subset \mathbb{C}[x,y,z]$,
we define the \emph{fat point subscheme} of $\P^2$ supported at $P_1,\dots,P_r$ with multiplicities $m_1,\dots, m_r$ to be the scheme $X=m_1P_1+\ldots +m_rP_r$ associated to the ideal
\[
I(X)=I_1^{m_1}\cap\ldots\cap I_r^{m_r}.
\]
We will indicate by $I(X)_d$ the homogeneous component of degree $d$ of $I(X)$. The vector space $I(X)_d$ is the linear system of curves of degree $d$ in $\P^2$ containing $X$, that is, having multiplicity at least $m_i$ at $P_i$ for all $i\in\{1,\ldots, r\}$.
	
The \emph{virtual dimension} of such a system is
\begin{equation}
	\vdim I(X)_d=\binom{d+2}{2}-\sum_{i=1}^r\binom{m_i+1}{2},
\end{equation} while its \emph{expected dimension} is
\begin{equation}\label{def: expected dimension}
\edim I(X)_d = \max\left\{\vdim I(X)_d, 0\right\}.
\end{equation}
In general $\dim I(X)_d\geq\edim I(X)_d$. If either the conditions given by $X$ are independent or $\dim I(X)_d=0$, then $\dim I(X)_d= \edim  I(X)_d$ and the system is called \emph{nonspecial}. Otherwise it is called \emph{special}.
\end{definition}

Classifying the special linear systems is a very hard task, even if the base points are in general position. A conjectural answer comes from the celebrated Segre-Harbourne-Gimigliano-Hirschowitz conjecture (see \cite{Se-SHGH}, \cite{Ha-SHGH}, \cite{G-SHGH}, \cite{Hi-SHGH}).

\begin{conjecture}[SHGH Conjecture]\label{conj: SHGH}
Let $X=m_1P_1+\ldots +m_rP_r$ be a fat point scheme. Assume that $P_1,\dots,P_r$ are in general position. If the linear system $I(X)_d$ is special, then its general element is non-reduced, namely the linear system has some multiple fixed component.
\end{conjecture}

Mathematicians have been working on interpolation problems for over a century. A nice survey of the known results and the techniques applied to get them is \cite{Ci}.

In \cite{CHMN} Cook II, Harbourne, Migliore and Nagel focused on a subtler problem about special linear systems of plane curves. Namely, they drop the hypothesis of generality of some of the points, and they propose a classification problem analogous to the SHGH Conjecture (see \cite[Problem 1.4]{CHMN}), although this problem seems too difficult to be solved in full generality. In the same spirit, we focus on a simplified version, and we consider degree $d$ curves containing a general point of multiplicity $d-1$ and a bunch of (not necessarily general) simple points.

\begin{definition}\label{def: unex curve}
Let $d\in\N$. We say that a finite set of distinct points $Z\subset\P^2$ \emph{admits an unexpected curve of degree $d$} if
\begin{equation}\label{eq: unex curve}
\dim I(Z+(d-1)P)_d>\max\left\{\dim I(Z)_d -\binom{d}{2}, 0\right\}
\end{equation}
for a general $P\in\P^2$.
\end{definition}

We want to stress that in the definition of an unexpected curve we do not take into account the number of conditions that $Z$ imposes on curves of degree $d$. Compare inequality (\ref{eq: unex curve}) with equation (\ref{def: expected dimension}).

Recently, unexpected curves and hypersurfaces have been intensively studied.
In the paper \cite{DiMMO} Di~Marca, Malara and Oneto present a way to produce families of unexpected curves using supersolvable arrangements of lines. In \cite{BMSzSz} Bauer, Malara, Szemberg and Szpond consider the existence of special linear systems in $\mathbb{P}^3$ and exhibit there a quartic surface with unexpected postulation properties. In \cite{HMNT} Harbourne, Migliore, Nagel and Teitler construct new examples both in the projective plane and in higher dimensional projective spaces. Moreover they introduce two new methods for constructing unexpected hypersurfaces.

One of the purposes of this paper is to classify all unexpected plane curves in low degrees. By \cite[Theorem 1.2]{CHMN} and \cite[Corollary 6.8]{CHMN} unexpected conics cannot exist. For $d=3$ we recover the following result of Akesseh \cite{Sol}.

\begin{thm}\label{thm:nounexpectedcubics}
No set of points $Z\subset\P^2$ admits an unexpected cubic over $\mathbb{C}$.
\end{thm}

Things become more complicated for $d=4$. In this case there exists a configuration of nine points in $\P^2$ which admits an unexpected quartic. It was observed by Di Gennaro, Ilardi and Vall{\`e}s in \cite[Proposition 7.3]{DiGIV} and is discussed in \cite[Example 6.14]{CHMN} and by Harbourne in {\cite[Example 4.1.10]{Ha}}.

\begin{example}[An unexpected quartic]\label{ex:UnexpectedQuartic}
Let $Z_1,Z_2,Z_3,Z_4\in\p^2$ be four general points. The lines joining any two of them determine three intersection points $Z_5,Z_6,Z_7$. Take a line through any two of the points $Z_5$, $Z_6$ and $Z_7$ (in Figure \ref{fig: Brian's config}, through $Z_6$ and $Z_7$). Call $Z_8$ and $Z_9$ the two intersection points with the previous lines and define $Z=\{Z_1,\dots,Z_9\}$.

Note that the construction of the points $Z_8$ and $Z_9$ depends on the choice of two points among $Z_5$, $Z_6$ and $Z_7$. Nevertheless, the three possible choices provide three projectively equivalent configurations of nine points. For instance, choosing the pair $(Z_5,Z_7)$ 
is the same as choosing $(Z_6,Z_7)$ and then considering the projective linear transformation swapping $Z_1$ and $Z_4$ and fixing $Z_2$ and $Z_3$. Therefore, we conclude that the configuration $Z$ is unique up to projective equivalence.
\end{example}
 
\begin{figure}[ht]\label{fig: Brian's config}
\centering
\begin{tikzpicture}[line cap=round,line join=round,>=triangle 45,x=1.0cm,y=1.0cm,scale=0.45]
\clip(-7,-7) rectangle (7,7);
\draw [line width=0.3pt,domain=-8.:2.] plot(\x,{(-0.--1.*\x)/1.});
\draw [line width=0.3pt,domain=-3.:2.5] plot(\x,{(-0.--2.*\x)/-2.});
\draw [line width=0.3pt,domain=-8.:8.] plot(\x,{(-12.-4.*\x)/-2.});
\draw [line width=0.3pt,domain=-1.:1.78] plot(\x,{(-6.--5.*\x)/-1.});
\draw [line width=0.3pt,domain=-3.3:5.] plot(\x,{(-4.--1.*\x)/-3.});
\draw [line width=0.3pt,domain=-8.:4.9] plot(\x,{(--18.-4.5*\x)/-7.5});
\draw [line width=0.3pt] (0.,-3.9) -- (0.,8.);
\draw [line width=1.2pt,] (0.005557527706670098,1.332101727856274)-- (0.007934814585178657,1.3331168080070637)-- (0.01,1.334)-- (0.012200674579804852,1.3349257486377006)-- (0.04435661922284132,1.3486142993799834)-- (0.07683655323255811,1.361033891769418)-- (0.11159778170207525,1.372822480568722)-- (0.14368380476106135,1.3824803982100569)-- (0.1663342151883824,1.388877842376568)-- (0.18579825381076956,1.3936697503622943)-- (0.20266873398728938,1.3977453422058248)-- (0.22499803717845537,1.4022655440686491);
\draw [line width=1.2pt,] (0.22499803717845537,1.4022655440686491)-- (0.22778920007735112,1.4029324590975905)-- (0.23102497299553998,1.4035746735699042)-- (0.23566867764149155,1.4044144924952526)-- (0.24134980566579264,1.405303712533841)-- (0.24824126096483617,1.4065634409218413)-- (0.2552809196036441,1.4076255648568219)-- (0.2621229737894328,1.4087370899050573)-- (0.2684710168426736,1.4097251121701555)-- (0.2744238509898768,1.4106884338786563)-- (0.281167102949156,1.4116270550304997)-- (0.2867988298602023,1.412516275069088)-- (0.29285046623391436,1.4131831900980294)-- (0.2975929731063744,1.4137760034570883)-- (0.30463263174518235,1.4145911218257945)-- (0.3100420536465821,1.4151839351848683)-- (0.318,1.416)-- (0.3256281048784691,1.4166906691391432)-- (0.33434740136793994,1.41711057860181)-- (0.3424738844983533,1.4178268947440062)-- (0.34941474091065206,1.41832090587657)-- (0.3568990095687531,1.418592611999472)-- (0.36710033945586773,1.4189137192356291)-- (0.37473281145373316,1.4190619225753938)-- (0.381772470092533,1.4192101259151586)-- (0.38851572205181206,1.419234826471786)-- (0.3949625673315623,1.4193089281416684)-- (0.40207632764025225,1.419284227585041)-- (0.407930359560945,1.4193583292549232)-- (0.4156369332286551,1.419284227585041)-- (0.4232200041132648,1.418963120348884)-- (0.4317169955930882,1.4186914142259819);
\draw [line width=1.2pt,] (0.4317169955930882,1.4186914142259819)-- (0.4394729703740897,1.4183950075464524)-- (0.446,1.418)-- (0.45278657039625475,1.4175304880644912)-- (0.45987563014831667,1.4168882735921773)-- (0.4673104976931621,1.4160978557800987)-- (0.475091173030791,1.41530743796802)-- (0.4833905600575952,1.4142453140330393)-- (0.4910477326120799,1.4133560939944507)-- (0.49910011407261007,1.41229397005947)-- (0.5060903715981623,1.411207145567862)-- (0.5131547307935967,1.410021518849744)-- (0.5207378016782064,1.4087617904617438)-- (0.5285925786857161,1.407329158177351)-- (0.5364226551365996,1.4056248197700565)-- (0.5434129126621519,1.4040933852591542)-- (0.5512676896696628,1.4024137474084872);
\draw [line width=1.2pt,] (0.5512676896696628,1.4024137474084872)-- (0.5586284555446259,1.4005859062180555)-- (0.5659151197497082,1.3985851611312312)-- (0.5737945973138466,1.3965103143745248)-- (0.5825138938033164,1.3942131626081713)-- (0.59,1.392)-- (0.5967661149773221,1.389865864641739)-- (0.6055595131366726,1.387074701742836)-- (0.6143529112960245,1.3841847366174236)-- (0.6233933150216503,1.380973664255854)-- (0.6320385098412379,1.3776390891111474)-- (0.6370774233932274,1.3755395417978136)-- (0.6444628898248179,1.372476672776009)-- (0.6522682657190741,1.3692408998578123)-- (0.659110319904862,1.366474437515537)-- (0.6659276735340225,1.3634115684937325)-- (0.6732884394089855,1.3599534905658883)-- (0.6796858835754803,1.3568906215440837)-- (0.6860833277419751,1.3536548486258868)-- (0.6925301730217247,1.3503202734811801)-- (0.6988041144050825,1.347183302789493)-- (0.703849221656292,1.3442631171942738)-- (0.7106096189491798,1.3405558025497777)-- (0.717089631100888,1.3369419496190085)-- (0.724223875248682,1.3327673264058784)-- (0.7299561936905776,1.3293092429979871);
\draw [line width=1.2pt,] (0.7299561936905776,1.3293092429979871)-- (0.7363738980331317,1.325383851021462)-- (0.743103141421444,1.3209911504763028)-- (0.7495519996685764,1.3170034506906263)-- (0.7549416251601414,1.3133895977598569)-- (0.7596770186556204,1.3100872838748436)-- (0.7653158753837894,1.3059749684708755)-- (0.770923578207383,1.301987268685199)-- (0.777154359122487,1.2974076447125862)-- (0.7825751385186274,1.2930149441674268)-- (0.7872482242049553,1.28952570685496)-- (0.7925443879827937,1.2850706985006493)-- (0.7979963212835096,1.2809272291920948)-- (0.8037286397254052,1.2761606817920284)-- (0.8100840362588112,1.2707087484912987)-- (0.8161590476510375,1.2656618159500517)-- (0.8203648247687327,1.26192334740098)-- (0.8259413736877508,1.2567829531460062)-- (0.8317983077479485,1.2516737127956084)-- (0.8354433145832842,1.248059859864839)-- (0.8402721697924899,1.2434490819876505)-- (0.845381410142875,1.2381529182097988)-- (0.8511448824893462,1.2325140614816157)-- (0.8565968157900621,1.2267817430397057)-- (0.8614256709992677,1.2214544253572783)-- (0.8672826050594654,1.2150990288238563)-- (0.8726410766464547,1.2089617096224636)-- (0.87775031699684,1.2033228528942805)-- (0.88,1.2)-- (0.8868986834066963,1.19058488931153);
\draw [line width=1.2pt,] (0.8868986834066963,1.19058488931153)-- (0.897380936434565,1.178047292552677)-- (0.9078631894624335,1.1642764895552482)-- (0.9175233049979202,1.1511222896771074)-- (0.9247170080562614,1.1406400366492138)-- (0.9312941079953162,1.1305688523675121)-- (0.9374601391881802,1.1207032024589065)-- (0.9458870484850941,1.1073434682076695)-- (0.9516948524706488,1.097767898134204)-- (0.9568961286433775,1.0887028168045774)-- (0.9618001890348074,1.0800835591468996)-- (0.9657686098888901,1.0725389255245832)-- (0.9705650291689956,1.0639161492906655)-- (0.975,1.055)-- (0.98,1.045);
\draw [line width=1.2pt,] (0.98,1.045)-- (0.9835835373833466,1.0369252264131446)-- (0.9874786189553951,1.0285240700812552)-- (0.9905930364157587,1.0218906070079106)-- (0.9920453143191821,1.0186955956203716)-- (0.9933330958545789,1.0159785036293787)-- (0.9943000043479833,1.0139082517016138)-- (0.9954508034138442,1.0112032807014395)-- (0.9962376840029507,1.009310159677628)-- (0.996734312705552,1.008068587921122)-- (0.99719783282798,1.0069594504853099)-- (0.9980421016224025,1.0048405013542063)-- (0.9985828750985684,1.003510640050571)-- (0.999,1.0025);
\draw [line width=1.2pt,] (0.999,1.0025)-- (0.9993876171054071,1.0015231047009434)-- (0.9994769642680115,1.0012839375781342)-- (0.9996,1.001)-- (0.9997160179067515,1.0006937739074944)-- (0.9998243390868056,1.0004173681377009)-- (0.9999550715454915,1.0000849341713278);
\draw [line width=1.2pt] (-1.4579327674272156,-0.20799619816480194)-- (-1.196892165419736,-0.18051824005871248)-- (-0.8566503900477577,-0.1376849524227495)-- (-0.5883567394231182,-0.10088528009158544)-- (-0.4479690885763671,-0.07960953966472273)-- (-0.2605545190744047,-0.048970904872420624)-- (-0.09367063856469428,-0.018694277487243893)-- (3.036052656470094E-4,7.178715648877732E-5);
\draw [line width=1.2pt] (0.005557527706670098,1.332101727856274)-- (0.0026353136800025692,1.3310058975962709)-- (-0.0047371145375631986,1.3278246788787134)-- (-0.014722272305354273,1.3231402838765332)-- (-0.028035815995742373,1.3178395211109082)-- (-0.04795732732120207,1.3100755970419338)-- (-0.0725147911614706,1.2994124877428548)-- (-0.12583033765679041,1.2742087748541224)-- (-0.1910596993784429,1.239739269459747)-- (-0.26,1.2)-- (-0.31525163534964684,1.1650784281866293);
\draw [line width=1.2pt] (-0.31525163534964684,1.1650784281866293)-- (-0.37097411688511367,1.128658505614377)-- (-0.43325218448357666,1.0856829969791193)-- (-0.5071846273051438,1.032509910023631)-- (-0.6153517973445795,0.9480156896560026)-- (-0.7369943387357294,0.8467683049051415)-- (-0.8590010793526012,0.7371443379626611)-- (-0.9963041874497972,0.6085820112826112)-- (-1.1344788423768943,0.47137696652538863)-- (-1.225892163864254,0.37724839786500375)-- (-1.3182105677425775,0.28040458203172314)-- (-1.411434054011865,0.18084551902554685)-- (-1.496340100042257,0.08528652397526854);
\draw [line width=1.2pt] (-2.,2.)-- (-2.0061010188881174,1.9253915519147102)-- (-2.0126412894781573,1.8349178087523623)-- (-2.018091514969857,1.7291834342132328)-- (-2.0213616502648772,1.6299893302641528)-- (-2.022451695363217,1.5318852714134141)-- (-2.023541740461557,1.4250608517759433)-- (-2.022451695363217,1.3280468380235462)-- (-2.020271605166537,1.212502057599343)-- (-2.018091514969857,1.098047322273481)-- (-2.0104611992814774,0.9857726771443025)-- (-2.0006507933964173,0.86368762613005)-- (-1.9875702522163372,0.7307021241323821)-- (-1.971565998789021,0.6045255139060233)-- (-1.9535937676221757,0.47682808193088827)-- (-1.9318379088412576,0.32737479117480434)-- (-1.9006229810251576,0.1703542451905643)-- (-1.871299867016094,0.029414116566156074);
\draw [line width=1.2pt] (-2.,2.)-- (-1.9915032644961579,2.084383751349717)-- (-1.9833271569059092,2.1736920034894816)-- (-1.9738931866094687,2.262371324276149)-- (-1.9606856281944516,2.3604846153592707)-- (-1.9462202070732426,2.4579689750892895)-- (-1.9317547859520336,2.5497929526414413)-- (-1.9166604334777284,2.6365854793688177)-- (-1.9009371496503273,2.7246358688023884)-- (-1.8799591206362682,2.8330274741138224)-- (-1.856891311291082,2.9452908129272193)-- (-1.8245963782078216,3.079084107129487)-- (-1.8015285688626357,3.175968906379405)-- (-1.7707714897357207,3.286694391236454)-- (-1.7338629947834232,3.3974198760935033)-- (-1.693878791918434,3.520448192601336)-- (-1.6600460048788277,3.6188708458076015)-- (-1.6246753638828757,3.7126799371448236)-- (-1.5800775991488494,3.8234054220018727)-- (-1.5324041265021315,3.9479715924660574)-- (-1.483192799899068,4.055621369410411)-- (-1.4401328891213871,4.147892606791285)-- (-1.387845854605632,4.247853113953899)-- (-1.3386345280025684,4.341662205291121)-- (-1.2894232013995048,4.437009150584692)-- (-1.2371361668837497,4.524666826096523)-- (-1.1833112784116488,4.615400209521049)-- (-1.131024243895894,4.699982177120184)-- (-1.0618208158603357,4.807631954064537)-- (-1.,4.9)-- (-0.9326410835272935,4.989098720913587)-- (-0.8803540490115384,5.0644535647746345)-- (-0.8126884749323259,5.150573386330117)-- (-0.7496364627221507,5.22900393810386)-- (-0.6327595620398746,5.367410794175172)-- (-0.5758589656550823,5.428924952429088)-- (-0.5158826613575985,5.496590526508396)-- (-0.46051991892915195,5.555028976849616)-- (-0.40054361463166815,5.61808098905988)-- (-0.3359537484651472,5.68420870918284)-- (-0.27751529812400916,5.739571451611365)-- (-0.21600113987017963,5.801085609865281)-- (-0.14064629600923848,5.868751183944589)-- (-0.06007250391701196,5.94316180229321)-- (-0.03801637879127646,5.962953579003855);
\draw [line width=1.2pt] (-0.03801637879127646,5.962953579003855)-- (-0.016667409910643367,5.983278631602183)-- (0.,6.);
\draw [line width=1.2pt] (0.,6.)-- (0.007440940058896936,6.006056733359961)-- (0.02369806458035484,6.018249576751073)-- (0.015815822388132825,6.012214735072644)-- (0.030471866464295633,6.024653898532262)-- (0.03823094862226417,6.031550860450466)-- (0.04697531105426047,6.039309942608445)-- (0.10654089183360715,6.092910903332557);
\draw [line width=1.2pt] (0.10654089183360715,6.092910903332557)-- (0.33600770332278884,6.293800992945117)-- (0.6469746574028202,6.568265175603173)-- (0.8864238567355238,6.781756386512943)-- (1.2016280018748116,7.072591807826762)-- (1.4834575904699323,7.348064338032917)-- (1.7475845185077155,7.617749970021379)-- (1.902982648729627,7.782253459436214)-- (2.042598156350875,7.9352234938736395)-- (2.1967822386804277,8.108225318534998)-- (2.3855666859422033,8.32311179548282);
\draw [line width=1.2pt] (-2.421731464592546,-0.31500782845340164)-- (-3.0250476480032638,-0.3718223569361096)-- (-3.349702096475399,-0.4015823480460995)-- (-3.7365819809046936,-0.436753246630633)-- (-4.0720582443259,-0.46651323774062287)-- (-4.423767230170713,-0.4989786825878846)-- (-4.813352568337275,-0.534149581172418)-- (-5.140712470546678,-0.5639095722824079)-- (-5.484305095179688,-0.5963750171296696)-- (-5.82519226607543,-0.6261350082396595)-- (-6.1417303533357614,-0.6558949993496493)-- (-6.477206616756968,-0.6829495367223675)-- (-6.828915602601781,-0.7127095278323573)-- (-7.180624588446594,-0.7397640652050754)-- (-7.535039028028675,-0.7695240563150653)-- (-7.911097097508898,-0.801989501162327)-- (-8.24116245345557,-0.8263385847977732)-- (-8.563111448190437,-0.8533931221704913);
\draw [line width=1.2pt] (-2.0109660587517055,-0.5151610559327003)-- (-2.0558826671296444,-0.5702165047296889)-- (-2.094592199432454,-0.6218292144668616)-- (-2.1333017317352634,-0.6669903354868877)-- (-2.169246297445015,-0.7121514565069137)-- (-2.2079558297478243,-0.7609991996510236)-- (-2.2533393967777364,-0.8180001288433829)-- (-2.2945651640560945,-0.8701846443856862)-- (-2.336834621645297,-0.9213254696171436)-- (-2.3791040792344997,-0.9740318303148698)-- (-2.423460917445391,-1.0309129522559806)-- (-2.5236551872864634,-1.157199479868335)-- (-2.628425682633679,-1.2946819678289427)-- (-2.7796334558870357,-1.4923321285818356)-- (-2.945962006465728,-1.7083432332298059)-- (-3.1004099462888,-1.9157138936918574)-- (-3.262072970033177,-2.1330197218031617)-- (-3.3661036050054247,-2.274292327407467)-- (-3.484909712561634,-2.4384607669399228)-- (-3.6,-2.6)-- (-3.725762094243768,-2.763557479435115)-- (-3.859688979125313,-2.9514871404788474)-- (-3.9863246867435835,-3.1288872847188864)-- (-4.113533575947183,-3.304133097118505)-- (-4.233625883936595,-3.473151900955701)-- (-4.371509644961476,-3.6679683116943633)-- (-4.504945542727486,-3.857447286522378)-- (-4.681080927778624,-4.1029693384121995)-- (-4.831418705928333,-4.323583356052334)-- (-4.991541783247549,-4.549534809603111)-- (-5.181910330727059,-4.8199648957426255)-- (-5.393190755487162,-5.123365392875917)-- (-5.467151397253818,-5.231090675449246)-- (-5.538700278962864,-5.333188517888146)-- (-5.61587660080633,-5.44412948053829)-- (-5.707523482995434,-5.575972363687731)-- (-5.796758605126942,-5.707815246837177)-- (-5.891621167392868,-5.842873810063439)-- (-5.958346528986699,-5.940148132387116)-- (-6.,-6.);
\draw [line width=1.2pt] (-6.,-6.)-- (-6.1454624122092385,-6.211324631902503)-- (-6.271532861591601,-6.3963371095677966)-- (-6.415613375171445,-6.602634208557416)-- (-6.556419331624474,-6.808931307547034)-- (-6.7414318092895,-7.08235682763645)-- (-6.957617103723639,-7.4007595645619615)-- (-7.323035817220222,-7.940187189248189)-- (-7.876393949596476,-8.753616065345515)-- (-8.377242704232836,-9.491814669233868);
\draw [line width=1.2pt] (-1.7805615966263373,-0.31883845947713346)-- (-1.7728575551123,-0.34329779264779836)-- (-1.7637310132304043,-0.37253997786127246)-- (-1.7551632392188286,-0.39824329989604584)-- (-1.7443603937259724,-0.43046558041789945)-- (-1.733930060146663,-0.4611978132855631)-- (-1.7249897742215405,-0.4855973436229204)-- (-1.7147456965990056,-0.5131632252254312)-- (-1.7029510756722306,-0.5450676309058872)-- (-1.6877385828668638,-0.5859893043239564)-- (-1.6760044575901407,-0.6152314895374305)-- (-1.6633544518415526,-0.6462810319357901)-- (-1.647303328447496,-0.687185507681995)-- (-1.6307344268794375,-0.7270544270801947)-- (-1.616236638007386,-0.7638166774343789)-- (-1.5981144019173221,-0.8062744877025915)-- (-1.5774032749572489,-0.8539100797108301)-- (-1.5539103713977513,-0.9079445221829706)-- (-1.53483945596237,-0.9493324663193912)-- (-1.5141454838941901,-0.9927492312468128)-- (-1.496697625091607,-1.0304852979594317)-- (-1.4715402472832315,-1.0828288743672576)-- (-1.4498371756034836,-1.1242580763242478)-- (-1.4241965694704295,-1.1725227466924206)-- (-1.3970476923883723,-1.22267275574685)-- (-1.3698988153063152,-1.2705603583777714)-- (-1.342749938224258,-1.3180708932714413)-- (-1.322809517759893,-1.351123675369965)-- (-1.3069438637870714,-1.3772758522482809)-- (-1.288462992126422,-1.4069149860437056)-- (-1.2698077726199175,-1.435682380609853)-- (-1.2492347268090058,-1.4672393407096875)-- (-1.2235704651450579,-1.5057488372600978)-- (-1.2036101240199142,-1.534682175771816)-- (-1.1823679261252658,-1.5645311262617532)-- (-1.157463280317747,-1.598775014247141)-- (-1.1292624313886463,-1.6366813501454094)-- (-1.107653988702711,-1.6645159542832648)-- (-1.086594913203706,-1.6912518240472576)-- (-1.063704613748266,-1.7198189177676881)-- (-1.0360531320060944,-1.7527809489835697)-- (-1.0062041815162008,-1.787574204155889)-- (-0.9820320252912561,-1.8150425635024572)-- (-0.9532818091752233,-1.8461733707619008)-- (-0.9263628170156257,-1.8749235868779752)-- (-0.8932176634041477,-1.9082518628851444)-- (-0.8664817936401936,-1.9349877326491371)-- (-0.8375484551285174,-1.962456091995705)-- (-0.8018395879780309,-1.9948687560246607)-- (-0.767412577597049,-2.0250839513058856)-- (-0.7403104630418079,-2.048157373157003)-- (-0.712260074903584,-2.070395205919518)-- (-0.6896442378245455,-2.0887168967177527)-- (-0.6690323356765611,-2.1047483761662082)-- (-0.6438400108290245,-2.1233563433831653)-- (-0.62,-2.14)-- (-0.5931690847152294,-2.158568342886023)-- (-0.569746277555396,-2.1742575033846205)-- (-0.5469682226226843,-2.1889005386985274)-- (-0.5247818054804327,-2.202951936221974)-- (-0.49386873092889544,-2.221144798278646)-- (-0.4659138453296584,-2.2371190186210903)-- (-0.44210042426364166,-2.250282959458845)-- (-0.41991400712139004,-2.2616719869252173)-- (-0.3925507593126129,-2.275575475001064)-- (-0.366814515427601,-2.287851959153128)-- (-0.34,-2.3)-- (-0.31741275992418744,-2.309742557400181)-- (-0.2953742522295508,-2.3190608525999403)-- (-0.2768855712776745,-2.326160506085471)-- (-0.25558661082111295,-2.333999706809078)-- (-0.2316252803074812,-2.342578454770761)-- (-0.21,-2.35)-- (-0.1893231782895881,-2.356481942846592)-- (-0.1615162021379661,-2.3649127813607937)-- (-0.13193431261496394,-2.3730478009796316)-- (-0.10338778922526687,-2.3799995450175473)-- (-0.07365799025464971,-2.386655470160233)-- (-0.04940084084578795,-2.391684391379151)-- (-0.029285155970146484,-2.3952342181219164)-- (-0.017181955492201847,-2.3970681960307965)-- (-0.00751229028756571,-2.3986985465594874)-- (0.,-2.4)-- (0.01246138989065485,-2.4016990969121728)-- (0.023423829207000698,-2.4030568302219972)-- (0.03343082656458246,-2.404263704275175)-- (0.04816474729710234,-2.4059231560982943)-- (0.06475926552826808,-2.4076328943402956)-- (0.0783365986264946,-2.4087391955557087)-- (0.09095848976595702,-2.409644351095592)-- (0.10398267225647802,-2.4104489337977104)-- (0.11861602015123444,-2.411152943662064)-- (0.13093619277740307,-2.4116558078508876)-- (0.14320607898468937,-2.412058099201947)-- (0.15633083431297515,-2.4122592448774767)-- (0.172573347611963,-2.412309531296359)-- (0.19052559915295136,-2.4121586720397117)-- (0.20968472474711544,-2.4117060942697703)-- (0.22617867014051654,-2.4110020844054167)-- (0.245,-2.41)-- (0.26383047326146847,-2.408524885856892)-- (0.28197433832250585,-2.407129203929118)-- (0.3023668020449538,-2.404880605267704)-- (0.3250078644288124,-2.402011703527279)-- (0.35960070644029346,-2.396862706207764)-- (0.38908959088648465,-2.391446380493149)-- (0.4251984289838615,-2.384024008217566)-- (0.4619090810495281,-2.3751974033493046)-- (0.5018294076127393,-2.3641641472639776)-- (0.5493727111076189,-2.3489181934006176)-- (0.5977184332268834,-2.3314655883201856)-- (0.6462647600022458,-2.311405122710501)-- (0.6982213659312491,-2.2873325639788793)-- (0.7455640647700321,-2.2630594005911604)-- (0.795314019481956,-2.234974748737602)-- (0.8524863464694681,-2.199267119952363)-- (0.908053836208208,-2.161352839950059)-- (0.9700406749420382,-2.114411350423402)-- (1.0234015134637162,-2.070679535394289)-- (1.0693399797098224,-2.0305586041749195)-- (1.1168832832047029,-1.9850213472409217)-- (1.1670344472288365,-1.935070787872807)-- (1.1955203083945443,-1.9045788801460863)-- (1.219793471782225,-1.8768954376047216)-- (1.2518939974923946,-1.8399614077289006)-- (1.2810538404037113,-1.8048615968170758)-- (1.308053694951227,-1.7719217742690554)-- (1.3369435393170686,-1.734661974993426)-- (1.3596234171369819,-1.704152139354686)-- (1.3836532876842706,-1.6720223124430922)-- (1.40876315241346,-1.6366525029857917)-- (1.4333330200516994,-1.6015526920739669)-- (1.4519629196894852,-1.5742828389809336)-- (1.4714028149636964,-1.5448529975240959)-- (1.490032714601482,-1.5154231560672582)-- (1.5092880260065173,-1.4859001832953334)-- (1.5,-1.5)-- (1.520389210688688,-1.4680023549301733)-- (1.5360214911594996,-1.4419485541454464)-- (1.5625284015230498,-1.3989031441532889)-- (1.5867697639922793,-1.3592560560026175)-- (1.6108298694489156,-1.3198887199809326)-- (1.6427049500444193,-1.263638577753485)-- (1.675517533010379,-1.2087946890817236)-- (1.7036426041240589,-1.1591070634474783)-- (1.7284864169411427,-1.115981954406435)-- (1.7547364833139105,-1.069106835883562)-- (1.7833303056128185,-1.020825463805003)-- (1.813799132652641,-0.9702003358003002)-- (1.833486682432217,-0.9378565040195178)-- (1.8522367298413367,-0.906450174609193)-- (1.8705180260652288,-0.8787938546806978)-- (1.8892680734743486,-0.8511375347522028)-- (1.9070806185130125,-0.8239499660089364)-- (1.9248931635516764,-0.7976998996361275)-- (1.9435964809711592,-0.7707398139943259)-- (1.9633273325837743,-0.7437608944423406)-- (1.9832595194169262,-0.7169833101108924)-- (2.0019836949268517,-0.6940310949696444)-- (2.0186945182314138,-0.6734949024748499)-- (2.036009347197586,-0.6539653860827412)-- (2.0513108239583895,-0.6370532275576163)-- (2.072249686894226,-0.6145036828574496)-- (2.096208578138116,-0.589940785951911)-- (2.1298315599677817,-0.5581298211070334)-- (2.1630518713563682,-0.528936214129139)-- (2.1960708475244184,-0.5017559593566168)-- (2.2228484318558244,-0.48162243730289656)-- (2.2522433740542045,-0.46048223914648984)-- (2.2834503332374223,-0.43954337621062084)-- (2.309019906245607,-0.4234365585676447)-- (2.3353948201359396,-0.40813508180681735)-- (2.3625750749084196,-0.39283360504599)-- (2.391164676224658,-0.3771294578440883)-- (2.418344930997138,-0.36303599240648415)-- (2.450357231062503,-0.3477345156456568)-- (2.484584218553774,-0.3322317036642923)-- (2.524851262661152,-0.3147155394775556)-- (2.56572231243014,-0.2986087218345794)-- (2.6051840156553703,-0.2839112507353637)-- (2.653303133363687,-0.26659642176916437)-- (2.7133010290836794,-0.24686557015651858)-- (2.7765202883322626,-0.22753738898494724)-- (2.826048752584337,-0.21324258832680593)-- (2.866315796691715,-0.202370486417797)-- (2.9087975282249987,-0.1919010549498625)-- (2.977855508869151,-0.17438489076312597)-- (3.025571956136394,-0.16351278885411705)-- (3.0805364713429646,-0.15143267562188495)-- (3.159057207352351,-0.1351245227583716)-- (3.232343227628072,-0.12042705165915586)-- (3.3100586227553186,-0.10613225100101453)-- (3.3797206090610885,-0.09344813210717082)-- (3.4487785897052476,-0.08177068931601311)-- (3.512400519394947,-0.07150259306861582)-- (3.588505232757899,-0.05922114461584651)-- (3.6712540083985687,-0.046537025722002794)-- (3.74152000036595,-0.0360675942540683)-- (3.8115846571127645,-0.02620216844774541)-- (3.88547468304981,-0.015531401759273714)-- (3.956747351119875,-0.005867311173488026)-- (4.,0.)-- (4.056609620506188,0.007622148602504496)-- (4.1365397030593405,0.017487574408827387)-- (4.247676744795713,0.03198371028750592)-- (4.375725945057187,0.04748652226887046)-- (4.51223122458121,0.0633920046913094)-- (4.629408322933691,0.07688146446730193)-- (4.905640245510354,0.1070817475478822)-- (5.273133169494088,0.14559636797820494)-- (5.643348934558216,0.18261794448467142)-- (6.027535105851179,0.22033804130258067)-- (6.497639275451477,0.2657418615463603)-- (6.927229266988155,0.3048589989871551)-- (7.514684848910693,0.3579465426568051)-- (8.098716413217973,0.4088154119643845);
\draw [line width=1.2pt] (3.036052656470094E-4,7.178715648877732E-5)-- (0.0013055303442626436,2.8052154786733233E-4)-- (0.002630458171969699,5.856850042979952E-4)-- (0.005358076966663793,0.001155865204319609)-- (0.010352705921614642,0.002299448619662259)-- (0.022030131842021718,0.0047274282664830376)-- (0.03959551073952203,0.008608414235823985)-- (0.061281302413794314,0.013283532960414035)-- (0.09473938099695726,0.020662334802839296)-- (0.12684561802120423,0.027984809913642994)-- (0.15562857787978374,0.03480034443985259)-- (0.1921282999705071,0.043587314572817025)-- (0.2177569628582834,0.05000856197767565)-- (0.255721179971165,0.05975308654820672)-- (0.2843914863665016,0.06741352204873981)-- (0.31244219871400214,0.07524293774413761)-- (0.34415414861689897,0.0842552148035883)-- (0.36730443531333073,0.09123972952466254)-- (0.3979461773154194,0.10058996697384265)-- (0.4312352757037172,0.11134837271356124)-- (0.4605251761468902,0.12120555074733538)-- (0.4947155022411325,0.13325947131435062)-- (0.5183164027905355,0.1419901147156934)-- (0.5436634320202044,0.15184729274946754)-- (0.5692357681763596,0.16215508463621423)-- (0.5940195300898136,0.17257552998620404)-- (0.6183526781502958,0.1833902624575448)-- (0.6389119351921383,0.19290948010158968)-- (0.659020578381009,0.20276665813536385)-- (0.679974122544202,0.21341241041183992)-- (0.7034060428987403,0.22580429136858432)-- (0.723176725697882,0.23690065749803293)-- (0.7442429233243172,0.2494615186496438)-- (0.7645768734396738,0.26224768672773946)-- (0.7831646948747645,0.2745832409528596)-- (0.8008512886039112,0.28703144864122576)-- (0.8164537932630638,0.2985221018920256)-- (0.8342530404554536,0.31243480460255246)-- (0.8505877926256846,0.32600954692334955)-- (0.8670351982591586,0.3407671506081999)-- (0.8816238217491227,0.35451087312386215)-- (0.8961561185074662,0.3689868431506047)-- (0.9097308608282444,0.3836317933722119)-- (0.9225733556379433,0.3985020505203054)-- (0.9337823752305338,0.41218944630434606)-- (0.9461742561872596,0.42852419847460027)-- (0.9564257213423699,0.4430564952329645)-- (0.9642551370377558,0.45522306926322276)-- (0.9726663621557098,0.46882481975162216)-- (0.9789940136402842,0.4793709055592628)-- (0.9878869292402266,0.49601661937456587)-- (0.9965518213632428,0.5134034094898653)-- (1.00470366066319,0.5314172641667001)-- (1.01,0.545)-- (1.0163328579862076,0.5614593572511682)-- (1.021520392086214,0.57662291846648)-- (1.0253967911938815,0.5892212155664182)-- (1.0288171433477313,0.6014204715817417)-- (1.0318384544169428,0.6132206865124531)-- (1.0353728183092352,0.6286122712046854)-- (1.0391922115476804,0.6481652843507464)-- (1.0419284932707398,0.6656090803352763)-- (1.0442359008877171,0.6842627407300443)-- (1.04583193080892,0.6997933395786954)-- (1.0469982603667218,0.7160605676217502)-- (1.0475507322625228,0.7279080204983891)-- (1.0478576610935233,0.7462623645922495)-- (1.047734889561123,0.7647394802185087)-- (1.047182417665322,0.7795334498727573)-- (1.046200245406122,0.7964145355778112)-- (1.0447269870173195,0.8133570070490669)-- (1.0432694280382249,0.8269089584957893)-- (1.0413821226564024,0.8414266922021371)-- (1.0390592852633902,0.8558718372399531)-- (1.0360831498535945,0.8730027630134419)-- (1.0328892484382026,0.8887545040848293)-- (1.029695347022811,0.9033448264597087)-- (1.026574034275947,0.9162656094583579)-- (1.0238882535402767,0.9262828457157378)-- (1.0209121181304799,0.9370259686584351)-- (1.0171375073668352,0.9494386309773625)-- (1.014117818755902,0.95918003029433)-- (1.0110037648758952,0.9685221919343648)-- (1.0078098634605035,0.9777046585036299)-- (1.005,0.985)-- (1.0025399261251071,0.9921570624082992)-- (1.0000956830303942,0.9997304687014905)-- (1.0000581673361866,0.9998272762384144)-- (1.0000184252260595,0.9999291790848941)-- (0.9999847972867213,1.0000117203905425)-- (0.9999550715454915,1.0000849341713278);
\draw [line width=1.2pt] (-1.7805615966263373,-0.31883845947713346)-- (-1.787266990469414,-0.297649414932973)-- (-1.792094874036429,-0.28102003820211285)-- (-1.7970568654803059,-0.26519530873242336)-- (-1.8015368323981797,-0.2511411586067161)-- (-1.8060334998111258,-0.23393377682782404)-- (-1.8083471699002662,-0.22511735359381138)-- (-1.811825653588652,-0.2124200627969742)-- (-1.8152451457855412,-0.1988649283971524)-- (-1.8188080757435487,-0.1835652879892094)-- (-1.8221687853342348,-0.16897890946545055)-- (-1.8254785874928212,-0.15449852502160932)-- (-1.8291019844924927,-0.1400193068223624)-- (-1.833100537387191,-0.1233586697610905)-- (-1.837432303023114,-0.10536518173491682)-- (-1.8414308559178123,-0.08803811919119403)-- (-1.8450804697244878,-0.07286501419684081)-- (-1.8487857481977452,-0.05804390030378533)-- (-1.8524910266710024,-0.044281437403090956)-- (-1.8561963051442598,-0.030518974502396573)-- (-1.859901583617517,-0.01622718610552164)-- (-1.8633467225857363,-0.002389405080179365)-- (-1.8666138055378099,0.010025510137722708)-- (-1.869322799897498,0.020402262952604503)-- (-1.871299867016094,0.029414116566156074);
\draw [line width=1.2pt] (-2.421731464592546,-0.31500782845340164)-- (-2.3166344405164723,-0.30331876168655963)-- (-2.170305625126053,-0.28714557682759007)-- (-2.010884231516491,-0.26789178532881674)-- (-1.877647994345215,-0.25248875212979804)-- (-1.7459520604938379,-0.23785587059073032)-- (-1.6096352166827632,-0.2232229890516626)-- (-1.4579327674272156,-0.20799619816480194);
\draw [line width=1.2pt] (-2.0109660587517055,-0.5151610559327003)-- (-1.963660956215121,-0.4527839017255142)-- (-1.919502145095105,-0.3952941287578435)-- (-1.8669432832706956,-0.3342726887883281)-- (-1.82362547331279,-0.2817383235201345)-- (-1.7747777301687684,-0.22183071400377338)-- (-1.7213217093696507,-0.1610014489563913)-- (-1.67155231069461,-0.10662377262615581)-- (-1.6236262230816079,-0.055011062888983144)-- (-1.5626074124038152,0.012619097182656953)-- (-1.496340100042257,0.08528652397526854);
\begin{scriptsize}
\draw [fill=black] (0.,0.) circle (3.pt);
\draw[color=black] (0.3,-0.5) node {$Z_2$};
\draw [fill=black] (-2.,2.) circle (3.pt);
\draw[color=black] (-2.5,1.8) node {$Z_1$};
\draw [fill=black] (1.,1.) circle (3.pt);
\draw[color=black] (1.6,1.1) node {$Z_3$};
\draw [fill=black] (0.,6.) circle (3.pt);
\draw[color=black] (-0.6,6.2) node {$Z_4$};
\draw [fill=black] (-6.,-6.) circle (3.pt);
\draw[color=black] (-5.4,-6.2) node {$Z_7$};
\draw [fill=black] (1.5,-1.5) circle (3.pt);
\draw[color=black] (2.1,-1.6) node {$Z_6$};
\draw [fill=black] (4.,0.) circle (3.pt);
\draw[color=black] (3.9,0.5) node {$Z_9$};
\draw [fill=black] (0.,1.33) circle (3.pt);
\draw[color=black] (-0.4,1.8) node {$Z_5$};
\draw [fill=black] (0.,-2.4) circle (3.pt);
\draw[color=black] (0.3,-2.9) node {$Z_8$};
\draw [fill=white] (-1.8035052604104964,-0.24410849141067362) circle (6pt);
\draw [fill=black] (-1.8035052604104964,-0.24410849141067362) circle (3pt);
\draw[color=black] (-2.3,0.1) node {$P$};
\end{scriptsize}
\end{tikzpicture}
\caption{A configuration of nine points in $\P^2$ admitting an unexpected quartic.}
\end{figure}
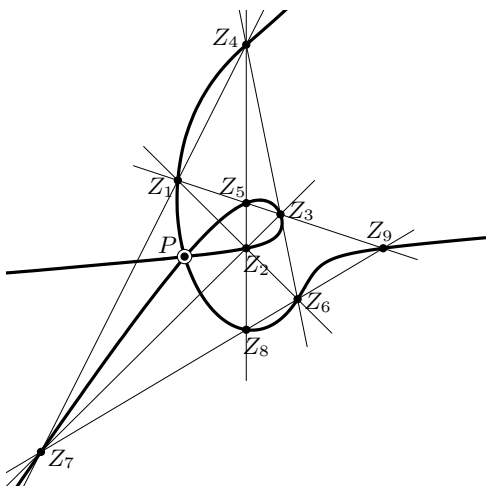

In this paper we analyze the geometry of this configuration and prove the following result.

\begin{thm}\label{thm: Main result}
Up to projective equivalence, the configuration of points $Z\subset\P^2$ in Example \ref{ex:UnexpectedQuartic} is the only one which admits an unexpected curve of degree four.
\end{thm}

Our paper is organized as follows. 
In Section \ref{sec: special config} we use plane geometry arguments and B\'{e}zout's Theorem to prove Theorem \ref{thm:nounexpectedcubics}. Moreover, we describe in detail the configuration of nine points of Example \ref{ex:UnexpectedQuartic} and give a geometric proof of the existence of an unexpected quartic.
Then we turn our attention to all possible configurations admitting an unexpected quartic. In Section \ref{sec: uniqueness proof, geom arg} we prove some tight necessary conditions on a set $Z$ with this property, and we achieve the results with a degeneration technique.
Finally, in Section \ref{sec: semi stable approach} we show how these necessary conditions lead to a unique configuration of points.
Here, the stability of vector bundles turns out to be a powerful tool to prove Theorem \ref{thm: Main result}.

\section{Unexpected cubics and quartics in \texorpdfstring{$\mathbb{P}^2$}{P2}}\label{sec: special config}

Let us fix the notation.
Given two points $A,B\in\mathbb{P}^2$, we denote by $AB$ the line joining them. We call a line \emph{simple} if it contains only two points of $Z$. For $k\geq 3$, we say that a line is  \emph{$k$-rich} if it contains exactly $k$ points of $Z$.

A standard tool to prove that a linear system is empty is to degenerate some of the points to a special position. If the degenerated linear system is empty, then the original one is empty as well. 

As shown in \cite[Corollary 5.5]{CHMN}, the most interesting case is $|Z|=2d+1$.
We will repeatedly use the following simple but useful result.

\begin{proposition}\label{pro:two and four}
Let $d$ be a positive integer. Let $Z$ be a set of $2d+1$ distinct points, and $P\in\P^2$ a general point. If there are a $d$-rich line $L_Q$ and a simple line $L_R$ such that $L_Q\cap L_R\notin Z$, then $I(Z+(d-1)P)_d=0$.
\end{proposition}
\begin{proof}
Assume by contradiction that $I(Z+(d-1)P)_d\neq 0$. Consider $\{Q_1,\ldots,Q_d\}=Z\cap L_Q$, $\{R_1,R_2\}=Z\cap L_R$ and $\{S_1,\ldots,S_{d-1}\}=Z\setminus(L_Q\cup L_R)$. Let $C$ be the degree $d$ curve defined by a non-zero element of $I(Z+(d-1)P)_d$.

We specialize the $(d-1)$-ple point $P$ to a general point on $L_R$. By B\'ezout's Theorem, $L_R$ and $L_Q$ are irreducible components of $C$. We are left with a degree $d-2$ curve $C'$ passing through $d-1$ simple points and a $(d-2)$-ple point $P$. Again by B\'{e}zout's Theorem, for all $i\in\{1,\ldots,d-1\}$ the line joining $S_i$ and $P$ is an irreducible component of $C'$. Hence $C'$ is a curve of degree $d-2$ having $d-1$ lines as components, which is impossible.\qedhere
\end{proof}

We now consider the problem of existence of unexpected cubics and we show that they cannot appear if the ground field is $\mathbb{C}$.

\begin{proof}[Proof of Theorem \ref{thm:nounexpectedcubics}]
Consider a set of points $Z\subset\P^2$. If $|Z|<7$, then any unexpected cubic is reducible by \cite[Corollary 5.5]{CHMN}. By \cite[Theorem 5.9]{CHMN}, this implies that some subset of $Z$ admits an unexpected conic, but this is impossible.

Assume now that $|Z|\ge 7$. Let $W$ be any subset of seven points of $Z$. Observe that $I(W+2P)_3$ contains $I(Z+2P)_3$ for every $P\in\p^2$, hence $\dim I(W+2P)_3\ge\dim I(Z+2P)_3$. Moreover, if $Z$ admits an unexpected cubic, then there exists a subset of seven points of $Z$ admitting an unexpected cubic: indeed, any seven smooth points of an irreducible plane cubic impose independent conditions on the system of plane cubics. For this reason, in order to conclude it is enough to prove that no subset $Z$ of seven points admits an unexpected cubic over $\mathbb{C}$. Therefore, for the rest of the proof we assume that $|Z|=7$.

If an unexpected cubic exists, then \cite[Theorem 1.2]{CHMN} implies that $Z$ contains no subset of four or more collinear points. On the other hand, \cite[Corollary 6.8]{CHMN} shows that the points of $Z$ cannot be in linearly general position. Suppose then that $L$ is a $3$-rich line and consider $\{Z_1,Z_2,Z_3\}=Z\cap L$. Let $Z_4$ and $Z_5$ be two points of $Z\setminus L$. By Proposition \ref{pro:two and four}, $Z_4Z_5$ must meet $L$ at a point of $Z$, and we assume that $Z_3=L\cap Z_4Z_5$. Since there cannot be four collinear points, we have that $Z\setminus(L\cup Z_1Z_2)=\{Z_6,Z_7\}$. Again by Proposition \ref{pro:two and four}, the lines $Z_4Z_6$ and $Z_5Z_6$ meet $L$ at a point of $Z$. Up to relabeling, the only possibility is that $Z_1\in Z_4Z_6$ and that $Z_2\in Z_5Z_6$. A similar argument is used to show that $Z_2\in Z_4Z_7$ and that $Z_1\in Z_5Z_7$. Hence, up to projective equivalence, $Z$ is the configuration described in Figure \ref{fig: seven points}.

\begin{figure}[ht]
\centering
\begin{tikzpicture}[line cap=round,line join=round,>=triangle 45,x=1.0cm,y=1.0cm, scale=0.3]
\clip(-5.7,-2.7) rectangle (7.3,9.1);
\draw [line width=0.8pt,domain=-1.:7] plot(\x,{(--10.27168882144967-2.2978427523614644*\x)/2.2245073453712085});
\draw [line width=0.8pt,domain=-5.5:7.2] plot(\x,{(--0.6415341939736247-1.8089400390930679*\x)/5.4757103886060525});
\draw [line width=0.8pt,domain=-5.3:5.] plot(\x,{(-21.478855519620723-3.226757907571418*\x)/-4.302343876761898});
\draw [line width=0.8pt,domain=0.6:4.55] plot(\x,{(--2.8481990858513844-2.7378551943030214*\x)/-1.0511408335270547});
\draw [line width=0.8pt,domain=-0.47:1.2] plot(\x,{(--4.614152570979933-5.035697946664486*\x)/1.1733665118441539});
\draw [line width=0.8pt,domain=-5.6:3.] plot(\x,{(-14.686487978559086-0.9289151552099537*\x)/-6.526851222133107});
\begin{scriptsize}
\draw [fill=black] (-0.21023924747911593,4.834681968980781) circle (3.2pt);
\draw[color=black] (-1.2,5) node {$Z_5$};
\draw [fill=black] (2.0142680978920926,2.5368392166193163) circle (3.2pt);
\draw[color=black] (2.3,1.6) node {$Z_2$};
\draw[color=black] (3.5,4.5) node {$L$};
\draw [fill=black] (-4.512583124241014,1.6079240614093626) circle (3.2pt);
\draw[color=black] (-4.7,2.4) node {$Z_4$};
\draw [fill=black] (0.963127264365038,-0.2010159776837052) circle (3.2pt);
\draw[color=black] (0.1,-0.5) node {$Z_1$};
\draw [fill=black] (0.37939673098534854,2.304160894737458) circle (3.2pt);
\draw[color=black] (-0.4,2.8) node {$Z_7$};
\draw [fill=black] (6.4051927350687485,-1.9988411778268924) circle (3.2pt);
\draw[color=black] (6.8,-1.5) node {$Z_6$};
\draw [fill=black] (4.152796000585067,8.106958405028912) circle (3.2pt);
\draw[color=black] (4.8,7.7) node {$Z_3$};
\end{scriptsize}
\end{tikzpicture} 
\caption{The considered configuration of seven points.}\label{fig: seven points}
\end{figure}
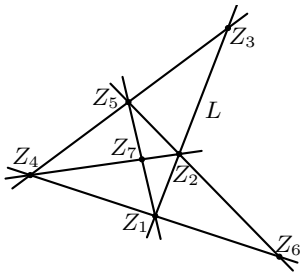

It is easy to check that in this case it is not possible that $Z_3\in Z_6Z_7$ over $\mathbb{C}$.
Hence the line $Z_6Z_7$ is simple and $L$ is a $3$-rich line, and therefore $Z$ does not admit an unexpected cubic by Proposition~\ref{pro:two and four}.
\end{proof}

Actually, a stronger version of Theorem \ref{thm:nounexpectedcubics} holds. In \cite{Sol}, Akesseh proves that unexpected cubics exist only if the characteristic of the ground field is 2.

Now we turn our attention to the case $d=4$. The proof of existence of the unexpected quartic in Example \ref{ex:UnexpectedQuartic} presented in \cite{CHMN} uses splitting types.  Here we give a new, simpler proof.

\begin{proposition}\label{thm:SufficientConditionQuartic}
The configuration $Z$ of nine points of Example \ref{ex:UnexpectedQuartic} admits an unexpected curve of degree four.
\end{proposition}
\begin{proof}
Let $P=[a,b,c]$ be a general point. Up to projective equivalence, we can assume that
\[Z_1=[-1,0,1],\ Z_2=[0,-1,1],\ Z_3=[1,0,1],\ Z_4=[0,1,1].\]
By construction, the remaining points have coordinates
\[Z_5=[0,0,1],\ Z_6=[1,-1,0],\ Z_7=[1,1,0],\ Z_8=[0,1,0],\ Z_9=[1,0,0].\]
Let $L_1$  be the linear form defining the line $Z_1Z_3$, let $L_2$ define $Z_2Z_4$, and let $L_3$ define the line $Z_6Z_7$. Furthermore, for every $j$ define $M_j$ to be the linear form defining the line $PZ_j$.

By using reducible quartics it is easy to see that $I(Z+2P)_4$ is nonspecial.
One can check that
\[G_1=L_1L_2M_6M_7,\ G_2=L_1L_3M_2M_4,\ G_3=L_2L_3M_1M_3\]
are linearly independent and thus form a basis of $I(Z+2P)_4$. Since each $G_i$ is singular at $P$, we have $G_i(P)=(G_i)_x(P)=(G_i)_y(P)=0$ for every $i$. The existence of an unexpected quartic is equivalent to the fact that the three additional conditions that the triple point $P$ imposes on $G_1$, $G_2$, $G_3$ (given by the three second order partials in $x$ and $y$) are linearly dependent. This means that
\[
\det
\begin{pmatrix}
(G_1)_{xx} & (G_2)_{xx} & (G_3)_{xx} \\
(G_1)_{xy} & (G_2)_{xy} & (G_3)_{xy} \\
(G_1)_{yy} & (G_2)_{yy} & (G_3)_{yy}
\end{pmatrix}(P)=0.
\]
This condition can be directly checked 
 by exploiting the facts that $G_1$, $G_2$, $G_3$ are completely reducible with pairwise common factors, and that $M_j(P)=0$ for every $j$.
\end{proof}

\section{Geometric conditions on unexpected quartics}\label{sec: uniqueness proof, geom arg}

We now focus on the proof of Theorem \ref{thm: Main result}. As \cite[Corollary 5.5]{CHMN} suggests, the most significant case is $|Z|=9$. Hence throughout this section $Z$ will indicate a set of nine points, and $P\in\P^2$ a general point. 
If an unexpected quartic exists, then \cite[Theorem 1.2]{CHMN} shows that $Z$ does not contain any subset of five or more collinear points. On the other hand, by \cite[Corollary 6.8]{CHMN}, the points of $Z$ cannot be in linearly general position. In this section we aim to provide further necessary conditions for the sets $Z$ giving rise to unexpected quartics.

For instance, the presence of a 4-rich line imposes a precise behavior on the configuration. The next propositions show how such a line has to intersect the other lines.

\begin{proposition}\label{pro:four and four}
If there are two $4$-rich lines $L_Q,L_R$ such that $L_Q\cap L_R\notin Z$, then $I(Z+3P)_4=0$.
\end{proposition}
\begin{proof}
Assume by contradiction that $I(Z+3P)_4\neq 0$. By hypothesis there exists a unique point $S\in Z\setminus(L_R\cup L_Q)$. Set $Z\cap L_R=\{R_1,R_2,R_3,R_4\}$. By Proposition \ref{pro:two and four}, for any $i\in\{1,2,3,4\}$ the lines $SR_i$ and $L_Q$ meet at a point of $Z$, say $Q_i$ (see Figure \ref{fig:four and four}).
\begin{figure}[ht]
\centering
\begin{tikzpicture}[line cap=round,line join=round,>=triangle 45,x=1.0cm,y=1.0cm,scale=0.5]
\clip(-7,-3.5) rectangle (6.5,4.);
\draw [line width=0.8pt,domain=-7.:2.5] plot(\x,{(-0.--6.66*\x)/-7.98});
\draw [line width=0.8pt,domain=-1.1:0.85] plot(\x,{(-0.--6.96*\x)/-2.26});
\draw [line width=0.8pt,domain=-1.8:1.65] plot(\x,{(-0.--7.34*\x)/4.66});
\draw [line width=0.8pt,domain=-7.:3.5] plot(\x,{(-0.--5.36*\x)/9.6});
\draw [line width=1.5pt,domain=-7.:7.] plot(\x,{(--21.7984-4.16*\x)/11.5});
\draw [line width=1.5pt,domain=-7.:7.] plot(\x,{(--20.96-4.*\x)/-14.24});
\begin{scriptsize}
\draw [fill=black] (0.,0.) circle (3.0pt);
\draw[color=black] (-0.7,0.05) node {$S$};
\draw[color=black] (-11.12,9.53) node {$f$};
\draw[color=black] (-2.86,9.53) node {$g$};
\draw[color=black] (5.74,9.53) node {$h$};
\draw[color=black] (-12.44,-7.01) node {$i$};
\draw [fill=white] (5.24,0.) circle (4pt);
\draw[color=black] (-16.62,7.77) node {$j$};
\draw[color=black] (-16.62,-5.79) node {$k$};
\draw [fill=black] (-4.008720997760018,3.345624291363624) circle (3.1pt);
\draw[color=black] (-3.5,3.6) node {$R_1$};
\draw [fill=black] (-0.697416476024372,2.1477958730662077) circle (3.1pt);
\draw[color=black] (-0.2,2.4) node {$R_2$};
\draw [fill=black] (0.9786594422114233,1.5414936278609115) circle (3.1pt);
\draw[color=black] (1.7,1.7) node {$R_3$};
\draw [fill=black] (2.060177994801922,1.1502660470977397) circle (3.1pt);
\draw[color=black] (3,1.2) node {$R_4$};
\draw [fill=black] (1.3195243865495305,-1.1012571947894578) circle (3.1pt);
\draw[color=black] (2.1,-1.28) node {$Q_1$};
\draw [fill=black] (0.4379974553954493,-1.3488771192709412) circle (3.1pt);
\draw[color=black] (1.1,-1.7) node {$Q_2$};
\draw [fill=black] (-1.1373053133616513,-1.7913778970116998) circle (3.1pt);
\draw[color=black] (-0.75,-2.2) node {$Q_3$};
\draw [fill=black] (-5.305433682078973,-2.9622004724940934) circle (3.1pt);
\draw[color=black] (-4.7,-3.25) node {$Q_4$};
\draw[color=black] (-5.7,3.5) node {$L_R$};
\draw[color=black] (-6.5,-2.8) node {$L_Q$};
\end{scriptsize}
\end{tikzpicture}
\caption{Two 4-rich lines not intersecting in $Z$.}
\label{fig:four and four}
\end{figure}
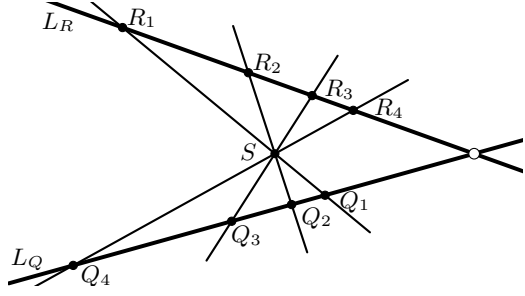
Up to projective equivalence, we assume that
\[
S=[0,0,1],\quad R_1=[1,0,0],\quad R_2=[0,1,0],\quad Q_3=[1,1,1].
\]
This choice of coordinates implies that $L_R$ is the line $z=0$ and that $R_3=[1,1,0]$. Since $Q_4\notin L_R$, $Q_4=[a,b,1]$ for some parameters $a$ and $b$. Therefore $L_Q$ is the line $(1-b)x+(a-1)y+(b-a)z=0$, $Q_1=[a-b,0,1-b]$ with $a\neq b$ and $b\neq 1$, $Q_2=[0,a-b,a-1]$ with $a\neq 1$ and finally $R_4=[a,b,0]$ with $a\neq 0$ and $b\neq 0$.

Now let $D$ be the quartic defined by a non-zero element of $I(Z+3P)_4$. We consider three different specializations of $P$ that put constraints on $a$ and $b$ and we show that there is no choice of $a$ and $b$ that satisfies all the constraints simultaneously.

First of all, observe that the lines $R_1Q_2$ and $R_2Q_1$ are simple. If we specialize $P$ to the point $R_1Q_2\cap R_2Q_1=[(1-a)(a-b),(1-b)(b-a),(1-a)(1-b)]$, then $D$ contains $R_1Q_2$, $R_2Q_1$ and the singular conic $R_3Q_3\cup R_4Q_4$. Moreover, since $D$ has multiplicity 3 at $P$, $P$ must be on either $R_3Q_3$ or $R_4Q_4$, in which case $b=2-a$ or $1/a+1/b=2$ respectively. 

Similarly as before, the lines $R_1Q_3$ and $R_3Q_1$ are simple. If we specialize $P$ to the point $R_1Q_3\cap R_3Q_1=[2b-a-1,b-1,b-1]$, then $D$ contains $R_1Q_3$, $R_3Q_1$ and the singular conic $R_2Q_2\cup R_4Q_4$. The conclusion now is that $P$ must be on either $R_2Q_2$ or $R_4Q_4$. The first case yields the condition $a=2b-1$, whereas the second one gives the condition $b=1/2$.

On one hand, if we assume that $a=2b-1$, then the only one compatible constraint between the two provided by the first specialization of $P$ is $1/a+1/b=2$. This gives the only one possible solution $(a,b)=(-1/2,1/4)$. On the other hand, if we assume that $b=1/2$, from the first specialization of $P$ we must have that $b=2-a$. Then another solution is $(a,b)=(3/2,1/2)$.

Finally, we observe that the lines $R_1Q_4$ and $R_4Q_1$ are simple as well. If we specialize $P$ to the point $R_1Q_4\cap R_4Q_1=[ab-2a+b,b(b-1),b-1]$, then $D$ contains $R_1Q_4$, $R_4Q_1$ and the singular conic $R_2Q_2\cup R_3Q_3$. The conclusion now is that $P$ must be on either $R_2Q_2$ or $R_3Q_3$. Since $R_2Q_2$ and $R_3Q_3$ are the lines $x=0$ and $x-y=0$ respectively, we reach a contradiction either if $(a,b)=(-1/2,1/4)$ or $(a,b)=(3/2,1/2)$.
\end{proof}

The previous result is important because it imposes a tight restriction on the set $Z$ admitting an unexpected quartic. Indeed, there is only one $Z$ having more than one 4-rich line.

\begin{proposition}\label{pro:Brian's example}
Assume that $I(Z+3P)_4\neq0$. If there are two $4$-rich lines $L_Q$, $L_R$, then the configuration of the points of $Z$ is the one described in Example \ref{ex:UnexpectedQuartic}.
\end{proposition}
\begin{proof}
By Proposition \ref{pro:four and four}, the $4$-rich lines $L_Q$ and $L_R$ meet at a point of $Z$. Then we can suppose that $L_R\cap Z=\{R_1,R_2,R_3,B\}$ and that $L_Q\cap Z=\{Q_1,Q_2,Q_3,B\}$. Let $\{S_1,S_2\}=Z \setminus (L_R \cup L_Q)$ and $L_S$ be the line containing $S_1$ and $S_2$. By Proposition \ref{pro:two and four}, $L_S$ meets $L_Q$ and $L_R$ at a point of $Z$.

Assume by contradiction that $B\in L_S$. Let $L_{ij}$ be the line joining $S_i$ and $R_j$ for $i\in\{1,2\}$ and $j\in\{1,2,3\}$. By Proposition \ref{pro:two and four}, each line $L_{ij}$ meets $L_Q$ at a point of $Z$. We show that the two cubics $C_1=L_{11}\cup L_{12}\cup L_{13}$ and $C_2=L_{21}\cup L_{22}\cup L_{23}$ never coincide when restricted to the line $L_Q$. This implies that one of the $L_{ij}$ meets $L_Q$ outside $Z$, hence contradicting Lemma \ref{pro:two and four}. Assume that $B=[1,0,0]$ and that the equations of $L_Q$, $L_R$ and $L_S$ are respectively $z=0$, $y=0$ and $y-z=0$. In particular, $S_i=[s_i,1,1]$ and $R_j=[r_j,0,1]$ for some parameters $s_i$ and $r_j$ for $i\in\{1,2\}$ and $j\in\{1,2,3\}$. With these assumptions, we obtain that $L_{ij}$ is defined by the linear form $l_{ij}=x+(r_j-s_i)y-r_jz=0$ for all $i$ and $j$. With a bit of work, one can see that $l_{11}l_{12}l_{13}$ and $l_{21}l_{22}l_{23}$ have the same roots on $L_Q$ if and only if $s_1=s_2$, and the latter condition is impossible since $S_1\neq S_2$.

The above argument implies that $B\notin L_S$. Up to relabeling, we assume that $Q_2\in L_S$ and $R_2 \in L_S$. Let $M_1$ be the line containing $R_1$ and $S_2$ and let $M_3$ be the line containing $R_3$ and $S_2$. By Proposition \ref{pro:two and four}, $M_1$ and $M_3$ meet $L_Q$ at a point of $Z$, and up to relabeling we assume that $Q_1\in M_1$ and that $Q_3\in M_3$. Now consider the line $N_1$ joining $S_1$ and $R_1$ and the line $N_3$ joining $S_1$ and $R_3$. Again by Proposition \ref{pro:two and four}, $N_1$ and $N_3$ meet $L_Q$ at a point of $Z$. In particular, $Q_i\notin N_i$ for $i\in\{1,3\}$ because $S_1 \neq S_2$. Moreover, $Q_2\notin N_i$ because $R_2\neq R_i$ for $i\in\{1,3\}$. Therefore, the only possibility is that $Q_3 \in N_1$ and that $Q_1 \in N_3$. Hence the obtained configuration is projectively equivalent to the one described in Example \ref{ex:UnexpectedQuartic}.\qedhere
\end{proof}

The following property of 4-rich lines is a further step toward the proof of uniqueness.

\begin{proposition}\label{pro:three and four}
Assume that there is exactly one 4-rich line $L_R$. If there is a 3-rich line $L_Q$ such that $L_Q\cap L_R\notin Z$, then $I(Z+3P)_4=0$.
\end{proposition}
\begin{proof}
Assume by contradiction that $I(Z+3P)_4\neq 0$. Suppose that $L_R\cap Z=\{R_1,R_2,R_3,R_4\}$ and that $L_Q\cap Z=\{Q_1,Q_2,Q_3\}$. By hypothesis there are only two points $S_1$, $S_2$ in $Z \setminus (L_R\cup L_Q)$. Moreover, by Proposition \ref{pro:two and four}, $L_S=S_1S_2$ must meet either $L_R$ or $L_Q$ at a point of $Z$.

Suppose that $L_S$ meets $L_R$ at a point of $Z$. If, in turn, $L_S$ meets $L_Q$ at a point of $Z$, then a 4-rich line distinct from $L_R$ appears, which is not allowed by hypothesis. Hence $L_S\cap L_Q\cap Z=\emptyset$ and we assume that $R_4\in L_S$. By Proposition \ref{pro:two and four}, the line $Q_iS_1$ meets $L_R$ at a point of $Z$ for $i\in\{1,2,3\}$. Moreover, $Q_iS_1\cap L_R\neq Q_jS_1\cap L_R$ for $i\neq j$ and $Q_iS_1\cap L_R$ is distinct from $R_4$ for $i\in\{1,2,3\}$. Therefore, we may assume that $Q_iS_1$ contains $R_i$ for all $i\in\{1,2,3\}$. Similarly the line $Q_iS_2$ meets $L_R$ at a point of $Z$ distinct from $R_4$ for $i\in\{1,2,3\}$. Suppose that $R_3\in Q_1S_2$ (the proof is similar if we consider $R_2\in Q_1S_2$). Consequently, $R_1\in Q_2S_2$ and $R_2\in Q_3S_2$. Up to projective equivalence, we assume that
\[
R_1=[0,0,1],\quad R_2=[0,1,0],\quad Q_3=[1,0,0],\quad S_1=[1,1,1].
\]
This choice of coordinates implies that $L_R$ and $Q_3S_1$ are the lines $x=0$ and $y-z=0$ respectively. Besides that, $R_3=[0,1,1]$. Since $S_2\in R_2Q_3$ and $R_2Q_3$ has equation $z=0$, $S_2=[1,a,0]$ for some $a\neq 0$. After little computation one verifies that $Q_2=[1,a,1]$, $Q_1=[1,1,1-a]$, $R_4=[0,1-a,1]$ for some $a\notin\{0,1\}$, and that $L_Q$ has equation $y-az=0$. Since $Q_1\in L_Q$ as well, we get the relation $a^2-a+1=0$. Now let $D$ be the quartic defined by a non-zero element of $I(Z+3P)_4$. Observe that the lines $R_2Q_1$ and $R_4Q_2$ are simple. If we specialize $P$ to the point $R_2Q_1 \cap R_4Q_2=[1,a^2,1-a]$, then $D$ contains $R_2Q_1$, $R_4Q_2$ and the singular conic $R_1S_2\cup R_3Q_3$. Since $D$ has multiplicity 3 at $P$, $P$ must be on either $R_1S_2$ or $R_3Q_3$. On one hand, if $P$ is on the line $R_1S_2$ of equation $y-ax=0$, then necessarily $a\in\{0,1\}$, a contradiction. On the other hand, if $P$ lies on the line $R_3Q_3$ of equation $y-z=0$, then necessarily $a^2+a-1=0$. This relation, combined with the relation $a^2-a+1=0$ obtained before, implies that $a=1$, again a contradiction.

Now suppose that $L_S$ meets $L_Q$ at a point of $Z$. In particular, suppose that $L_S$ contains the point $Q_3$. By Proposition \ref{pro:two and four}, the lines $Q_1S_1$ and	$Q_1S_2$ are not simple, hence up to labeling we assume that $Q_1S_1\cap L_R=R_1$ and $Q_1S_2\cap L_R=R_2$. Regarding the lines $Q_2S_1$ and $Q_2S_2$, we have three cases to consider:
\begin{figure}[ht]
\centering
\begin{tikzpicture}[line cap=round,line join=round,>=triangle 45,x=1.0cm,y=1.0cm,scale=0.57]
\clip(-2.5,-1.8) rectangle (6.,4.);
\draw [line width=1.5pt,domain=-3.:6.] plot(\x,{(-44.86361430994704--9.845934779981848*\x)/-7.89577881928462});
\draw [line width=1.5pt,domain=-3.:5.5] plot(\x,{(-0.-0.*\x)/-16.55656220688776});
\draw [line width=0.8pt,domain=-3.:4.] plot(\x,{(--2.081740768163536--1.040870384081768*\x)/4.731819148892283});
\draw [line width=0.8pt,domain=-3.:6.] plot(\x,{(-1.4414008302389505-0.7207004151194752*\x)/4.21720296036158});
\draw [line width=0.8pt,domain=2.05:3.1] plot(\x,{(--4.27664409157679-1.7615707992012433*\x)/-0.5146161885307032});
\draw [line width=0.8pt,domain=1.7:4.5] plot(\x,{(--2.2826482029236974-0.7207004151194752*\x)/-0.950060655748989});
\draw [line width=0.8pt,domain=-3.:3.8] plot(\x,{(-3.296710896589217--1.040870384081768*\x)/-0.4354444672182858});
\begin{scriptsize}
\draw[color=black] (-3.6415044167571677,10.49199461575136) node {$f$};
\draw[color=black] (-14.369272654589473,0.3382213574340811) node {$g$};
\draw [fill=black] (2.731819148892283,1.040870384081768) circle (2.5pt);
\draw[color=black] (2.2,1.3) node {$S_2$};
\draw [fill=black] (-2.,0.) circle (2.5pt);
\draw[color=black] (-1.8700371523918697,0.5262541955510677) node {$Q_1$};
\draw[color=black] (-14.369272654589473,-2.373410097516147) node {$i$};
\draw [fill=black] (2.21720296036158,-0.7207004151194752) circle (2.5pt);
\draw[color=black] (2.4,-1.1) node {$S_1$};
\draw[color=black] (-14.369272654589473,2.456064902541194) node {$h$};
\draw[color=black] (5.225728370233373,10.49199461575136) node {$l$};
\draw [fill=black] (5.597817290090197,-1.2984317156270906) circle (2.5pt);
\draw[color=black] (5.7,-0.8394579970880984) node {$R_1$};
\draw [fill=black] (3.5733969763352276,1.225994415436665) circle (2.5pt);
\draw[color=black] (3.7,1.7) node {$R_2$};
\draw [fill=black] (3.167263616110569,0.) circle (2.5pt);
\draw[color=black] (3.7,-0.3) node {$Q_2$};
\draw[color=black] (16.68583503020552,10.49199461575136) node {$j$};
\draw[color=black] (-1.0090446831193507,10.49199461575136) node {$k$};
\draw [fill=black] (4.031076328531085,0.6552741413506407) circle (2.5pt);
\draw[color=black] (4.6,0.6) node {$R_3$};
\draw [fill=black] (1.652068619846881,3.6218661998273936) circle (2.5pt);
\draw[color=black] (1.3,3.5545725357509577) node {$R_4$};
\draw [fill=black] (2.4277446546661463,0.) circle (2.5pt);
\draw[color=black] (2.108341843557011,0.4470824742386523) node {$Q_3$};
\end{scriptsize}
\end{tikzpicture}
\hspace{0.5cm}
\begin{tikzpicture}[line cap=round,line join=round,>=triangle 45,x=1.0cm,y=1.0cm,scale=0.57]
\clip(-2.5,-1.8) rectangle (6.,4.5);
\draw [line width=1.5pt,domain=-3.:6.] plot(\x,{(-44.86361430994704--9.845934779981848*\x)/-7.89577881928462});
\draw [line width=1.5pt,domain=-3.:5.5] plot(\x,{(-0.-0.*\x)/-16.55656220688776});
\draw [line width=0.8pt,domain=-3.:4.2] plot(\x,{(--1.7254680222576666--0.8627340111288333*\x)/4.47451105462693});
\draw [line width=0.8pt,domain=-3.:6.] plot(\x,{(-1.4414008302389505-0.7207004151194752*\x)/4.21720296036158});
\draw [line width=0.8pt,domain=1.8:4.] plot(\x,{(-5.078703779568466--1.8155840151563831*\x)/1.4613360895650032});
\draw [line width=0.8pt,domain=-3.:3.35] plot(\x,{(--2.413311885599999-0.8627340111288333*\x)/0.32277275993177934});
\draw [line width=0.8pt,domain=2.1:2.8] plot(\x,{(--3.6962375477668283-1.5834344262483087*\x)/-0.2573080942653503});
\begin{scriptsize}
\draw[color=black] (-3.6415044167571677,10.49199461575136) node {$f$};
\draw[color=black] (-14.369272654589473,0.3382213574340811) node {$g$};
\draw [fill=black] (2.47451105462693,0.8627340111288333) circle (2.5pt);
\draw[color=black] (2,1.1) node {$S_2$};
\draw [fill=black] (-2.,0.) circle (2.5pt);
\draw[color=black] (-1.9887947343604928,0.5064612652229639) node {$Q_1$};
\draw[color=black] (-14.369272654589473,-2.0369302819383814) node {$i$};
\draw [fill=black] (2.2002938996886905,2.938236316873201) circle (2.5pt);
\draw[color=black] (2.4052357984785693,3.336850302141815) node {$R_4$};
\draw [fill=black] (2.21720296036158,-0.7207004151194752) circle (2.5pt);
\draw[color=black] (2.5,-1.1) node {$S_1$};
\draw[color=black] (-14.369272654589473,2.456064902541194) node {$h$};
\draw [fill=black] (5.597817290090197,-1.2984317156270906) circle (2.5pt);
\draw[color=black] (5.730448093600022,-0.879043857744306) node {$R_1$};
\draw [fill=black] (3.678539049926583,1.094883600036908) circle (2.5pt);
\draw[color=black] (4.3,1.4) node {$R_2$};
\draw[color=black] (10.9656781653835,10.49199461575136) node {$j$};
\draw [fill=black] (2.7972838145587096,0.) circle (2.5pt);
\draw[color=black] (3.3,-0.3) node {$Q_2$};
\draw[color=black] (-0.9298729618069352,10.49199461575136) node {$k$};
\draw [fill=black] (1.2587446054073208,4.112336194264677) circle (2.5pt);
\draw[color=black] (0.9603518845269856,3.950431142313035) node {$R_3$};
\draw[color=black] (3.7808444562817893,10.49199461575136) node {$l$};
\draw [fill=black] (2.3343167778184943,0.) circle (2.5pt);
\draw[color=black] (1.9895842615883879,0.35) node {$Q_3$};
\end{scriptsize}
\end{tikzpicture}
\hspace{0.5cm}
\begin{tikzpicture}[line cap=round,line join=round,>=triangle 45,x=1.0cm,y=1.0cm,scale=0.57]
\clip(-2.5,-3.3) rectangle (7.2,2.5);
\draw [line width=1.5pt,domain=-3.:7.] plot(\x,{(-44.86361430994704--9.845934779981848*\x)/-7.89577881928462});
\draw [line width=1.5pt,domain=-3.:5.5] plot(\x,{(-0.-0.*\x)/-16.55656220688776});
\draw [line width=0.8pt,domain=-3.:4.5] plot(\x,{(--1.2108518337269663--0.6054259168634831*\x)/4.731819148892283});
\draw [line width=0.8pt,domain=2.3:7.] plot(\x,{(--11.784109180800417-3.410686196834817*\x)/4.074373509155029});
\draw [line width=0.8pt,domain=2.15:4.25] plot(\x,{(--2.628717238109572-0.7608321903542544*\x)/-0.4913706288773261});
\draw [line width=0.8pt,domain=-3.:7.] plot(\x,{(-5.6105205599426675-2.8052602799713338*\x)/8.806192658047312});
\draw [line width=0.8pt,domain=2.45:3.] plot(\x,{(-5.46542866772747--2.0466541350735996*\x)/0.2075568391210738});
\begin{scriptsize}
\draw[color=black] (-3.6415044167571677,10.49199461575136) node {$f$};
\draw[color=black] (-14.369272654589473,0.3382213574340811) node {$g$};
\draw [fill=black] (2.731819148892283,0.6054259168634831) circle (2.5pt);
\draw[color=black] (3.1,0.9715951279334045) node {$S_2$};
\draw [fill=black] (-2.,0.) circle (2.5pt);
\draw[color=black] (-1.8898300827199737,0.5262541955510677) node {$Q_1$};
\draw[color=black] (-14.369272654589473,-1.2452130688142276) node {$i$};
\draw [fill=black] (3.946425858448519,0.7608321903542544) circle (2.5pt);
\draw[color=black] (4.55,1.1) node {$R_2$};
\draw [fill=black] (6.806192658047312,-2.8052602799713338) circle (2.5pt);
\draw[color=black] (6.9,-2.4) node {$R_1$};
\draw[color=black] (-8.84704509304849,10.49199461575136) node {$h$};
\draw [fill=black] (3.455055229571193,0.) circle (2.5pt);
\draw[color=black] (3.6,-0.5) node {$Q_2$};
\draw[color=black] (9.956238718650201,10.49199461575136) node {$j$};
\draw[color=black] (-14.369272654589473,3.782191234524152) node {$k$};
\draw [fill=black] (2.5242623097712094,-1.4412282182101164) circle (2.5pt);
\draw[color=black] (2.75,-1.9) node {$S_1$};
\draw[color=black] (3.4839505013602308,10.49199461575136) node {$l$};
\draw [fill=black] (2.670421237309316,0.) circle (2.5pt);
\draw[color=black] (2.266685286181842,-0.35) node {$Q_3$};
\draw [fill=black] (5.373359412606985,-1.0185356251799798) circle (2.5pt);
\draw[color=black] (5.552311720647087,-0.6019428331508521) node {$R_3$};
\draw [fill=black] (3.342855034530084,1.5134772559539074) circle (2.5pt);
\draw[color=black] (3.5334328271804907,1.85) node {$R_4$};
\end{scriptsize}
\end{tikzpicture}
\caption{The cases (1), (2) and (3) of Proposition \ref{pro:three and four}.}\label{fig:three and four}
\end{figure}
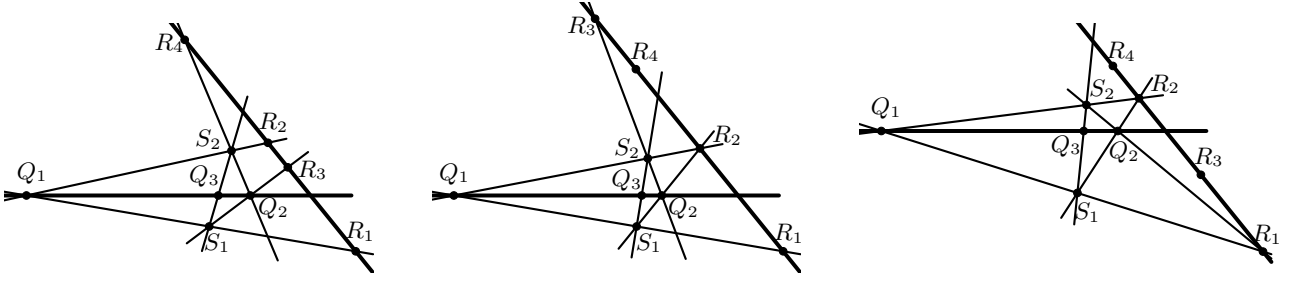
\begin{enumerate}
\item $Q_2S_1 \cap L_R=R_3$ and $Q_2S_2 \cap L_R=R_4$. Observe that the lines $Q_1R_3$, $Q_1R_4$ and $Q_2R_2$ are simple. Let $X_1=Q_1R_3\cap Q_2R_2$ and $X_2=Q_1R_4\cap Q_2R_2$. On one hand, we can specialize $P$ to $X_1$. In order for $D$ to exist, we need a conic having $L_R$ and $L_S$ as components. Noting that $X_1\notin L_R$, we have that $X_1\in L_S$. On the other hand, we can specialize $P$ to $X_2$. In order for $D$ to exist, we need a conic having $L_R$ and $L_S$ as components and, similarly as before, $X_2\in L_S$. Hence $Q_2R_2$ and $L_S$ should coincide, a contradiction.
\item $Q_2S_1 \cap L_R=R_2$ and $Q_2S_2 \cap L_R=R_3$. Note that the lines $Q_1R_3$, $Q_1R_4$ and $Q_2R_1$ are simple. Now repeat the argument from case $(1)$.
\item $Q_2S_1 \cap L_R=R_2$ and $Q_2S_2 \cap L_R=R_1$.
Up to projective equivalence, we assume that
\[
R_1=[1,0,0],\quad R_2=[0,1,0],\quad S_1=[0,0,1],\quad S_2=[1,1,1].
\]
This choice of coordinates implies that $L_R$ and $L_S$ are the lines $z=0$ and $x-y=0$ respectively. Moreover, we obtain that $Q_1=[1,0,1]$, $Q_2=[0,1,1]$. Therefore, $L_Q$ is the line $x+y-z=0$ and $Q_3=[1,1,2]$. Now we assume that $R_3=[1,a,0]$ and $R_4=[1,b,0]$ for some parameters $a$ and $b$ such that $a\neq 0$, $b\neq 0$ and $a\neq b$. Moreover, we exclude the case $\{a,b\}=\{-1,1\}$, which is not allowed by hypothesis and in particular coincides with the configuration of Example \ref{ex:UnexpectedQuartic}. Let $D$ be the quartic defined by a non-zero element of $I(Z+3P)_4$. We follow the same argument used in Proposition \ref{pro:four and four}. First of all, observe that, with the given constraints on the parameters $a$ and $b$, the lines $R_3S_1$, $R_3Q_2$, $R_3Q_3$ and $R_4S_2$ are simple.

If we specialize $P$ to the point $Y_1=R_3S_1\cap R_4S_2=[b-1,a(b-1),b-a]$, then $D$ contains $R_3S_1$, $R_4S_2$ and the singular conic $L_Q\cup L_R$. Moreover, $Y_1$ must be on either $L_R$ or $L_Q$, in which case $a=b$ (impossible by our assumption) or $ab=1$ respectively. Hence, we can rewrite $R_4=[a,1,0]$.

If we specialize $P$ to the point $Y_2=R_3Q_2\cap R_4S_2=[1,1+a-a^2,1-a^2]$, then $D$ contains $R_3Q_2$, $R_4S_2$ and the singular conic $Q_3R_2\cup Q_1R_1$. The conclusion now is that $Y_2$ must belong either to $Q_3R_2:2x-z=0$ or to $Q_1R_1:y=0$. The first case yields the condition $a^2+1=0$, whereas the second one gives the condition $a^2-a-1=0$.

Finally, if we specialize $P$ to the point $Y_3=R_3Q_3\cap R_4S_2=[a+2,2a+1,2(a+1)]$, then $D$ contains $R_3Q_3$, $R_4S_2$ and the singular conic $Q_1R_1\cup Q_2R_2$. Hence $Y_3$ must be on either $Q_1R_1:y=0$ or $Q_2R_2:x=0$, so either $a=-1/2$ or $a=-2$. Since neither $a=-1/2$ nor $a=-2$ is a root of either of the polynomials $a^2+1$ or $a^2-a-1$, we have that the initial constraints on $a$ and $b$ prevent the configuration from admitting an unexpected quartic.\qedhere
\end{enumerate}
\end{proof}

\begin{proposition}\label{pro:only one with four}
If $Z$ admits exactly one $4$-rich line $L$, then $I(Z+3P)_4=0$.
\end{proposition}
\begin{proof}
Suppose there is a $4$-rich line $L$, with $L \cap Z=\{Z_1,Z_2,Z_3,Z_4\}$. Let $R = \{R_1,\ldots,R_5\}$ denote the remaining five points of $Z$. First of all, by Propositions \ref{pro:Brian's example} and \ref{pro:three and four}, if three of the points of $R$ are collinear, then $I(Z+3P)_4=0$. For $i\in\{1,\ldots,5\}$, let $L_i$ be the line containing $Z_1$ and $R_i$. Since $|R|$ is odd and none of the lines $L_i$ can contain more than two points of $R$, one of the lines $L_i$ must contain exactly one of the points of $R$. So say $L_1$ contains only $R_1$.
For $j\in\{2,\ldots,5\}$, define $M_j$ to be the line $R_1R_k$. By Proposition \ref{pro:two and four}, the line $M_2$ intersects $L$ in $\{Z_1,Z_2,Z_3,Z_4\}$. Hence up to relabelings we may assume that $M_2\cap L=Z_2$. Now the point $R_3$ cannot belong to the line $M_2$ by hypothesis, so $M_2\neq M_3$ and we may assume that $M_3\cap L=Z_3$. By the same argument we obtain that $M_2,M_3,M_4$ are distinct lines and $M_4\cap L=Z_4$. Then the line $M_5$ contains two points of $R$ and $M_5\cap L\cap Z=\emptyset$. Proposition \ref{pro:two and four} implies that $I(Z+3P)_4=0$.
\end{proof}

\begin{corollary}\label{cor:final step}
The configuration of Example \ref{ex:UnexpectedQuartic} is the only configuration of nine points in $\P^2$ containing a $4$-rich line which admits an unexpected quartic. 
\end{corollary}

Corollary \ref{cor:final step} will be the first step in the proof of Theorem \ref{thm: Main result}. In the next section we will show that if a configuration $Z$ of nine points admits an unexpected quartic, then $Z$ has a 4-rich line.

\section{Unexpected curves and stability conditions}\label{sec: semi stable approach}

Let $Z\subset\P^2$ be a finite set of points. For us, the \textit{stability} (respectively, the \textit{semistability}) of $Z$ is the stability (respectively, the semistability) of its dual line arrangement $\mathcal{A}_Z$. The latter is defined in \cite[Section 6]{CHMN} in terms of the derivation bundle of $\mathcal{A}_Z$, but what we actually need are the following properties. The first one follows from \cite[Proposition 6.4]{CHMN}.
\begin{lemma}\label{lem:no unexpected if stable} If $Z\subset\P^2$ is semistable or stable, then $Z$ admits no unexpected curve.
\end{lemma}

The next results are proven in \cite[Lemma 6.5]{CHMN} and \cite[Proposition 6.7]{CHMN}.

\begin{lemma}\label{lem:stability properties} Let $Z\subset\P^2$ be a set of points and $P\in Z$. Consider $Z'=Z\setminus\{P\}$ and the line arrangement $\mathcal{A}=\{PQ\mid Q\in Z'\}$. We define the set $Z''\subset\P^2$ to be the dual of $\mathcal{A}$. Then
\begin{enumerate}
\item if $|Z|$ is odd, $Z'$ is stable and $|Z''| > \frac{|Z|+1}{2}$, then $Z$ is stable,
\item if $|Z|$ is odd and $Z'$ is stable, then $Z$ is semistable,
\item if $|Z|$ is even, $Z'$ is semistable and $|Z''| > \frac{|Z|}{2}$, then $Z$ is stable,
\item if $|Z|$ is even and $Z'$ is stable, then $Z$ is stable.
\end{enumerate}
\end{lemma}

\begin{lemma}\label{lem:if lingeneral then stable} If $Z\subset\P^2$ is a set of at least four points in linearly general position, then $Z$ is stable.
\end{lemma}

There are some configurations of points which will be useful for us.

\begin{definition}\label{def:badconfigurations}
For $n\geq 3$, the factors of the polynomial
\[
(x^n - y^n)(x^n - z^n)(y^n - z^n)\in\mathbb{C}[x,y,z]
\]
define the \emph{Fermat arrangement} of $3n$ lines in $\P^2$. Its dual is a configuration of $3n$ points in $\P^2$, called the \emph{dual Fermat configuration} and denoted by $F_n$.
\end{definition}

Since we are dealing with sets of nine points in the plane, for us the most interesting Fermat configuration is $F_3$, shown in Figure \ref{fig: dual Fermat configuration}.

\begin{figure}[ht]
\begin{minipage}[c]{.3\textwidth}
\centering
\begin{tikzpicture}[line cap=round,line join=round,>=triangle 45,x=1.0cm,y=1.0cm, scale=1.2]
\clip(-1.,-1.) rectangle (3.,3.);

\draw [line width=0.7pt] (-0.5,2) -- (2.5,2);
\draw [line width=0.7pt] (-0.5,1) -- (2.5,1);
\draw [line width=0.7pt] (-0.5,0) -- (2.5,0);

\draw [line width=0.7pt, dotted] (-0.5,2) -- (-0.8,2);
\draw [line width=0.7pt, dotted] (2.5,2) -- (2.8,2);
\draw [line width=0.7pt, dotted] (-0.5,1) -- (-0.8,1);
\draw [line width=0.7pt, dotted] (2.5,1) -- (2.8,1);
\draw [line width=0.7pt, dotted] (-0.5,0) -- (-0.8,0);
\draw [line width=0.7pt, dotted] (2.5,0) -- (2.8,0);

\draw [line width=0.7pt] (0.,-0.5) -- (0.,2.5);
\draw [line width=0.7pt] (1.,-0.5) -- (1.,2.5);
\draw [line width=0.7pt] (2.,-0.5) -- (2.,2.5);

\draw [line width=0.7pt, dotted] (0.,-0.5) -- (0.,-0.8);
\draw [line width=0.7pt, dotted] (0.,2.5) -- (0.,2.8);
\draw [line width=0.7pt, dotted] (1,-0.5) -- (1,-0.8);
\draw [line width=0.7pt, dotted] (1,2.5) -- (1,2.8);
\draw [line width=0.7pt, dotted] (2,-0.5) -- (2,-0.8);
\draw [line width=0.7pt, dotted] (2,2.5) -- (2,2.8);

\draw [line width=0.7pt] (-0.4,2.4) -- (2.4,-0.4);
\draw [line width=0.7pt] (2.4,2.4) -- (-0.4,-0.4);

\draw [line width=0.7pt, dotted] (-0.4,2.4) -- (-0.6,2.6);
\draw [line width=0.7pt, dotted] (2.4,2.4) -- (2.6,2.6);

\draw [line width=0.7pt, dotted] (-0.4,-0.4) -- (-0.6,-0.6);
\draw [line width=0.7pt, dotted] (2.4,-0.4) -- (2.6,-0.6);

\begin{scriptsize}
\draw [fill=black] (0,0) circle (1pt);
\draw[color=black] (-0.2,0.15) node {$P_5$};
\draw [fill=black] (0,1) circle (1pt);
\draw[color=black] (-0.2,1.15) node {$P_4$};
\draw [fill=black] (0,2) circle (1pt);
\draw[color=black] (0.2,2.15) node {$P_1$};
\draw [fill=black] (1,0) circle (1pt);
\draw[color=black] (0.8,-0.15) node {$P_9$};
\draw [fill=black] (2,0) circle (1pt);
\draw[color=black] (1.8,-0.15) node {$P_7$};
\draw [fill=black] (2,1) circle (1pt);
\draw[color=black] (2.2,0.85) node {$P_8$};
\draw [fill=black] (1,1) circle (1pt);
\draw[color=black] (1.2,0.85) node {$P_6$};
\draw [fill=black] (1,2) circle (1pt);
\draw[color=black] (1.2,2.15) node {$P_2$};
\draw [fill=black] (2,2) circle (1pt);
\draw[color=black] (2.2,1.85) node {$P_3$};
\end{scriptsize}
\end{tikzpicture}
\end{minipage}
\hspace{14pt}
\begin{minipage}[c]{.3\textwidth}
\centering
\begin{tikzpicture}[line cap=round,line join=round,>=triangle 45,x=1.0cm,y=1.0cm, scale=1.2]
\clip(-1.,-1.) rectangle (3.,3.);

\draw [line width=0.2pt, dotted] (-0.5,2) -- (2.5,2);
\draw [line width=0.2pt, dotted] (-0.5,1) -- (2.5,1);
\draw [line width=0.2pt, dotted] (-0.5,0) -- (2.5,0);
\draw [line width=0.2pt, dotted] (0.,-0.5) -- (0.,2.5);
\draw [line width=0.2pt, dotted] (1.,-0.5) -- (1.,2.5);
\draw [line width=0.2pt, dotted] (2.,-0.5) -- (2.,2.5);
\draw [line width=0.2pt, dotted] (-0.4,2.4) -- (2.4,-0.4);
\draw [line width=0.2pt, dotted] (2.4,2.4) -- (-0.4,-0.4);

\begin{scriptsize}
\draw [fill=black] (0,0) circle (1pt);
\draw[color=black] (-0.2,0.15) node {$P_5$};
\draw [fill=white] (0,1) circle (1.8pt);
\draw [fill=black] (0,1) circle (0.8pt);
\draw[color=black] (-0.2,1.15) node {$P_4$};
\draw [fill=white] (0,2) circle (1.8pt);
\draw[color=black] (0.2,2.15) node {$P_1$};
\draw [fill=white] (1,0) circle (1.8pt);
\draw[color=black] (0.8,-0.15) node {$P_9$};
\draw [fill=white] (2,0) circle (1.8pt);
\draw [fill=black] (2,0) circle (0.8pt);
\draw[color=black] (1.8,-0.15) node {$P_7$};
\draw [fill=white] (2,1) circle (1.8pt);
\draw[color=black] (2.2,0.85) node {$P_8$};
\draw [fill=black] (1,1) circle (1pt);
\draw[color=black] (1.2,0.85) node {$P_6$};
\draw [fill=white] (1,2) circle (1.8pt);
\draw [fill=black] (1,2) circle (0.8pt);
\draw[color=black] (1.2,2.15) node {$P_2$};
\draw [fill=black] (2,2) circle (1pt);
\draw[color=black] (2.2,1.85) node {$P_3$};
\end{scriptsize}
\end{tikzpicture}
\end{minipage}
\hspace{14pt}
\begin{minipage}[c]{.3\textwidth}
\centering
\begin{tikzpicture}[line cap=round,line join=round,>=triangle 45,x=1.0cm,y=1.0cm, scale=1.2]
\clip(-1.,-1.) rectangle (3.,3.);

\draw [line width=0.2pt, dotted] (-0.5,2) -- (2.5,2);
\draw [line width=0.2pt, dotted] (-0.5,1) -- (2.5,1);
\draw [line width=0.2pt, dotted] (-0.5,0) -- (2.5,0);
\draw [line width=0.2pt, dotted] (0.,-0.5) -- (0.,2.5);
\draw [line width=0.2pt, dotted] (1.,-0.5) -- (1.,2.5);
\draw [line width=0.2pt, dotted] (2.,-0.5) -- (2.,2.5);
\draw [line width=0.2pt, dotted] (-0.4,2.4) -- (2.4,-0.4);
\draw [line width=0.2pt, dotted] (2.4,2.4) -- (-0.4,-0.4);

\begin{scriptsize}
\draw [fill=white] (0,0) circle (1.8pt);
\draw [fill=black] (0,0) circle (0.8pt);
\draw[color=black] (-0.2,0.15) node {$P_5$};
\draw [fill=white] (0,1) circle (1.8pt);
\draw[color=black] (-0.2,1.15) node {$P_4$};
\draw [fill=black] (0,2) circle (1pt);
\draw[color=black] (0.2,2.15) node {$P_1$};
\draw [fill=white] (1,0) circle (1.8pt);
\draw[color=black] (0.8,-0.15) node {$P_9$};
\draw [fill=black] (2,0) circle (1pt);
\draw[color=black] (1.8,-0.15) node {$P_7$};
\draw [fill=white] (2,1) circle (1.8pt);
\draw [fill=black] (2,1) circle (0.8pt);
\draw[color=black] (2.2,0.85) node {$P_8$};
\draw [fill=black] (1,1) circle (1pt);
\draw[color=black] (1.2,0.85) node {$P_6$};
\draw [fill=white] (1,2) circle (1.8pt);
\draw [fill=black] (1,2) circle (0.8pt);
\draw[color=black] (1.2,2.15) node {$P_2$};
\draw [fill=white] (2,2) circle (1.8pt);
\draw[color=black] (2.2,1.85) node {$P_3$};
\end{scriptsize}
\end{tikzpicture}
\end{minipage}
\caption{The dual Fermat configuration $F_3$ and its corresponding twelve $3$-rich lines. The first eight $3$-rich lines are depicted on the left picture. The remaining four $3$-rich lines are obtained regarding the open circles as representing collinear points, and likewise the dotted circles as representing collinear points.}\label{fig: dual Fermat configuration}
\end{figure}
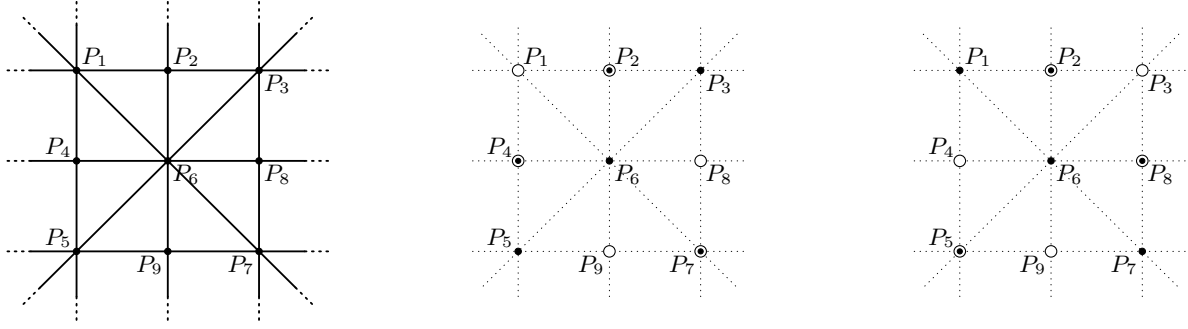

As pointed out in \cite[Section 1.1]{Ha}, $F_3$ has the peculiar feature to admit no simple lines and no $k$-rich lines for any $k\ge 4$. For our purpose, we need to know whether there are other configurations of nine points with similar properties. Our next task is to prove that $F_3$ is the only one, thereby solving \cite[Open problem 1.1.6]{Ha} for sets of nine points.

\begin{lemma}\label{lem:Fermat has a lot of quirks} Let $Y\subset\p^2$ be a set of nine points. Assume that every line that meets $Y$ at at~least two points contains exactly three points of $Y$. Then
	\begin{enumerate}
		\item every point of $Y$ is contained in exactly four 3-rich lines,
		\item $Y$ admits twelve 3-rich lines,
		\item for every 3-rich line $M$, there are two other 3-rich lines $M',M''$ such that $M\cap M'\notin Y$ and $M\cap M''\notin Y$.
	\end{enumerate}
	\begin{proof}
	\begin{enumerate}
		\item Let $P\in Y$. By hypothesis, for every $Q\in Y\setminus\{P\}$, the line $PQ$ is 3-rich, so there exists a unique $Q'\in T\setminus\{P,Q\}$ such that $Y\cap PQ=\{P,Q,Q'\}$. In this way the eight points of $Y\setminus\{P\}$ are partitioned in four pairs. Each pair defines a 3-rich line containing $P$.
		\item By hypothesis, each pair of points of $Y$ defines a 3-rich line. In this way, every such line is counted $\binom{3}{2}$ times, so the number of 3-rich lines is $\binom{9}{2}\cdot\frac{1}{3}=12$.
		\item Let $M$ be a 3-rich line and let $P\in M\cap Y$. As in part (1), the remaining eight points of $Y$ are partitioned into four pairs $(Q_1,Q_1'),\dots,(Q_4,Q_4')$ is such a way that $P\in Q_iQ_i'$  for every $i\in\{1,\dots,4\}$. We may assume $M=Q_1Q_1'$. Another 3-rich line meeting $M$ at a point of $Y$ is defined by the choice of a point among $\{P,Q_1,Q_1'\}$ and an index among $\{2,3,4\}$, so there are nine of them. Now the statement follows from part (2).\qedhere
	\end{enumerate}
	\end{proof}
\end{lemma}		
		
\begin{corollary}
	\label{cor: Fermat is the only one without simple lines}
	$F_3$ is the only configuration of nine points in $\p^2$ with no $k$-rich lines for every $k\ge 4$ and no simple lines.
\end{corollary}
	\begin{proof}
		Let $Y$ be such a configuration and let $P_1\in Y$. By Lemma \ref{lem:Fermat has a lot of quirks}(1), the point $P_1$ is contained in exactly four $3$-rich lines, call them $P_2P_3$, $P_4P_5$, $P_6P_7$ and $P_8P_9$ (see Figure \ref{fig: dual Fermat configuration} (middle)). By Lemma \ref{lem:Fermat has a lot of quirks}(3), there is a 3-rich line meeting $P_2P_3$ outside $Y$. Up to relabeling, we may assume that this line is $P_4P_6$ and that $P_4P_6\cap Y=\{P_4,P_6,P_8\}$. By Lemma \ref{lem:Fermat has a lot of quirks}(1), the point $P_4$ is contained in exactly four $3$-rich lines. Two of them are $P_4P_5$ and $P_4P_6$. Since $Y$ does not admit $4$-rich lines, the remaining two $3$-rich lines through $P_4$ must be $P_iP_j$ and $P_kP_l$ with $\{i,j,k,l\}=\{2,3,7,9\}$ and $\{i,j\}\neq\{2,3\}$. Without loss of generality, we can therefore suppose that $P_4\in P_2P_7\cap P_3P_9$. In a similar fashion, one may verify that $P_3\in P_1P_2\cap P_4P_9\cap P_5P_6\cap P_7P_8$. In the same way, $P_2\in P_6P_9\cap P_5P_8$ and $P_7\in P_5P_9$. Thus $Y=F_3$.
	\end{proof}

Now that we have a better understanding of the dual Fermat configuration, we can state our result on semistability of sets of nine points.

\begin{proposition}\label{pro: Fermat or die} If $Z\subset\P^2$ is a set of nine points containing no $k$-rich lines for $k \geq 4$, then either $Z=F_3$ or $Z$ is semistable.
\end{proposition}
\begin{proof}
 Our idea is to reduce the problem to the study of smaller subsets of $Z$.
Assume that $Z$ is not $F_3$. By Corollary \ref{cor: Fermat is the only one without simple lines}, $Z$ admits a simple line $L$. Assume $L\cap Z=\{Z_8,Z_9\}$. By Lemma \ref{lem:stability properties}(2), in order to conclude it is enough to show that there exists a stable subset of $Z$ with eight points. Since there are no 4-rich lines, the set $\{Z_jZ_9\mid j\in\{1,\ldots,8\}\}$ has at least five distinct elements $L,L_1,\ldots,L_4$. Up to relabeling, we can assume $Z_1\in L_1,\ldots,Z_4\in L_4$. Let $A\in\{Z_5,Z_6,Z_7\}$ and set $W_8=Z\setminus\{A\}$. All we need to do is to prove that $W_8$ is stable. In order to do that, we want to apply Lemma \ref{lem:stability properties}(3). Define $W_7=W_8\setminus\{Z_9\}$. We will prove that $W_7$ is semistable.
In turn, by Lemma \ref{lem:stability properties}(2) it is enough to check that there exists a stable subset $W_6\subset W_7$ with six elements.

We indicate by $S$ the configuration of six points in $\p^2$ given by the intersection points of four general lines. Now we want to show that $W_7$ contains at least a subset $W_6$ of six elements which is different from the configuration $S$. Consider one of the subsets of six points of $W_7$. If it is not $S$, then we are done. If it is $S$, then the seventh point of $W_7$ does not lie on any of the four 3-rich lines of $S$, because our hypothesis guarantees that $Z$ has no $k$-rich lines for $k\ge 4$. Therefore, if we replace one of the points of $S$ with the seventh one, the resulting subset of $W_7$ is not $S$. Call this subset $W_6$.

Since $W_6$ is not $S$, there exists a subset $W_5\subset W_6$ of five elements with at most one 3-rich line. By Lemma \ref{lem:stability properties}(4), it is enough to prove that $W_5$ is stable. If $W_5$ has no 3-rich lines, then it is stable by Lemma \ref{lem:if lingeneral then stable}. Otherwise $W_5$ has exactly one 3-rich line. Up to projective equivalence, we assume that $W_5=\{B_1,B_2,B_3,B_4,B_5\}$, where
\[
B_1=[1,0,0],\quad B_2=[0,1,0],\quad B_3=[0,0,1],\quad B_4=[1,1,1],\quad B_5=[1,a,0]
\]
for some parameter $a$. Using the \verb+Macaulay2+ lines one can verify that also in this case $W_5$ is stable.

{\small
\begin{verbatim}
KK = frac(QQ[a,b,c]); R = KK[x,y,z];
B1 = ideal(y,z); B2 = ideal(x,z); B3 = ideal(x,y);
B4 = ideal(x-y,x-z); B5 = ideal(y-a*x,z);
P = ideal(y-b*x,z-c*x);
W5 = intersect(B1,B2,B3,B4,B5);
m = j -> (J = intersect(W5,P^j); return binomial(j+3,2)-hilbertFunction(j+1,J))
m(1), m(2) -- = (0,2)
\end{verbatim}
}

In this way we check that the splitting type (see \cite[Section 1]{CHMN} for a definition) of $W_5$ is $(2,2)$, hence $W_5$ is stable.
\qedhere
\end{proof}

The last result of this section is an important step toward the proof of Theorem \ref{thm: Main result}.
\begin{lemma}\label{lem:halfway done} Up to projective equivalence, the only configuration of nine points $Z\subset\P^2$ admitting an unexpected quartic is the one presented in Example \ref{ex:UnexpectedQuartic}.
	\begin{proof}
		If $Z$ has a 4-rich line, then we conclude by Corollary \ref{cor:final step}. Assume then that $Z$ admits no 4-rich lines. Since the configuration $F_3$ does not admit an unexpected curve by \cite[Section 6]{CHMN}, Proposition \ref{pro: Fermat or die} and Lemma \ref{lem:no unexpected if stable} imply that $Z$ does not admit an unexpected quartic.
	\end{proof}
\end{lemma}

\begin{remark}\label{rmk:fermat}
It is interesting to point out that if $n\geq 5$ then the configuration $F_n$ admits unexpected curves of degrees $n+2,\ldots,2n-3$ by \cite[Proposition 6.12]{CHMN}.
\end{remark}

As a consequence, we can now complete the proof of our main result.

\begin{proof}[Proof of Theorem \ref{thm: Main result}]
	Thanks to Lemma \ref{lem:halfway done}, we know that the thesis holds if $|Z|=9$. If $|Z|<9$, then the unexpected curve is reducible by \cite[Corollary 5.5]{CHMN}. By \cite[Theorem 5.9]{CHMN}, this implies that some subset of $Z$ admits an unexpected cubic, and this contradicts Theorem \ref{thm:nounexpectedcubics}.
	
	Assume now that $|Z|>9$. Let $W$ be any subset of nine points of $Z$. Observe $I(W+3P)_4$ contains $I(Z+3P)_4$ for every $P\in\p^2$, hence $\dim I(W+3P)_4\ge\dim I(Z+3P)_4$. The latter equals 1 by \cite[Corollary 5.5]{CHMN}. Since $I(W+3P)_4$ is expected to be empty, $W$ admits an unexpected quartic too, so $\dim I(W+3P)_4=1$ for the same reason. It follows that $I(V+3P)_4=I(Z+3P)_4$ for every $V\subset Z$ such that $9\le |V|\le |Z|$. In particular, we consider a set of 10 points. This $V$ enjoys a peculiar property: if we remove any point from it, we get a subset $W$ admitting an unexpected quartic. By Lemma \ref{lem:halfway done}, this means that every time we remove a point from $V$, we get a configuration equivalent to Example \ref{ex:UnexpectedQuartic}. Such configuration has three 4-rich lines. It order to preserve this property, if we remove $Z_9$ (see Figure \ref{fig: Brian's config}), the tenth point of $V$ should lie in the intersection of two 3-rich lines, and this is not possible.
\end{proof}

We conclude by pointing out that there is a connection between existence and uniqueness of unexpected curves and de Jonqui\`{e}res transformations.
\begin{example}
Let $P$ be a general point and let $Z=\{Z_1, \dots,Z_9\}\subset\P^2$ be a set of nine points, not containing five collinear points. Let $\f$ be the degree four de Jonqui\`{e}res transformation with centers $P$ and $Z_1,\dots,Z_6$. In other words, $\f:\p^2\dashrightarrow\p^2$ is the birational map associated to the linear system of quartic plane curves containing $Z_1,\dots,Z_6$ and having multiplicity three at $P$. Let \begin{displaymath}
\xymatrix{
	X\ar[d] \ar[dr]^{\Phi} &\\
	\p^2 \ar@{-->}[r]_{\f}  & \p^2
}
\end{displaymath} be the resolution of indeterminacy of $\f$. If $Z$ admits an unexpected quartic $D$, then $\Phi(D)$ has degree $4\cdot 4-3^2-6=1$. Therefore the points $\Phi(Z_7)$, $\Phi(Z_8)$, $\Phi(Z_9)$ are collinear.
\end{example}
This phenomenon occurs every time that the linear system associated to such a de Jonqui\`{e}res-type transformation has no fixed components. However, this is not always the case.
\begin{example} Consider the Fermat configuration $F_{60}$ and a general point $P\in\p^2$. By Remark \ref{rmk:fermat}, there is an unexpected curve $C\in I(F_{60}+21P)_{22}$. Such $C$ is irreducible by \cite[Lemma 5.1]{CHMN}.  By B\'ezout's theorem, $C$ is an irreducible component of $I(F_{60}+30P)_{31}$, so $I(F_{60}+30P)_{31}\cong I(9P)_{9}$ has dimension ten. In this case the rational map $\f$ associated to the linear system of degree 31 curves containing $F_{60}$ and having multiplicity 30 at $P$ is not a birational transformation of $\p^2$, but rather is a rational map $\f:\p^2\dashrightarrow\p^9$.
\end{example}
\section*{Acknowledgements}
It is our great pleasure to thank the Organizers of PRAGMATIC 2017 in Catania: Elena Guardo, Alfio Ragusa, Francesco Russo and Giuseppe Zappal{\`a} for the stimulating atmosphere of the school and for the support.  We warmly thank Brian Harbourne and Adam Van Tuyl for helpful remarks. 
We thank the Organizers of IPPI 2018 in Torino: Enrico Carlini, Gianfranco  Casnati, Elena Guardo, Alessandro Oneto and Alfio Ragusa  for the invitation and the opportunity to continue our work on the paper and for the support.
We also thank Michela Di Marca, Grzegorz Malara and Alessandro Oneto for discussions and for sharing their ideas with us.
We thank the knowledgeable referees for many helpful remarks which improved
our paper.
%************************************************************************

\end{document}